\documentclass[a4paper,11pt]{article}

\usepackage{amsmath}
\usepackage{amsthm}
\usepackage{verbatim}

\usepackage[top=3.5cm, bottom=3.5cm, left=2.5cm, right=2.5cm]{geometry}

\usepackage[utf8]{inputenc}
\usepackage[T1]{fontenc}
\usepackage[french,english]{babel} 
\usepackage{enumerate}

\usepackage{xcolor}
\usepackage[]{pdfpages}

\usepackage{titling}
\usepackage[colorlinks=false,linkcolor=blue]{hyperref}
\usepackage{csquotes}
\usepackage[nottoc, notlof, notlot, numbib]{tocbibind}
\usepackage[style=alphabetic,giveninits,doi=false,isbn=false,url=false,eprint=false,clearlang=true,sorting=nyt,maxbibnames=99]{biblatex}
\DeclareSourcemap{
      \maps[datatype=bibtex]{
      \map[overwrite=false]{
       \step[fieldsource=note]
       \step[fieldset=addendum, origfieldval, final]
       \step[fieldset=note, null]
       }
   }
   }
\DefineBibliographyStrings{english}{in = {}}
\usepackage{mathtools}
\usepackage{thmtools}
\usepackage{bbm}
\usepackage{amssymb}
\usepackage{mathrsfs}
\usepackage{float}
\usepackage{stmaryrd}

\usepackage{setspace}
\usepackage{array}
\usepackage{tabularx}
\usepackage{ltablex} 

\usepackage{todonotes}
\usepackage{mathtools}
\mathtoolsset{showonlyrefs}

\DeclareUnicodeCharacter{221E}{$\infty$}
\DeclareUnicodeCharacter{00B4}{\'}
\DeclareUnicodeCharacter{03BB}{$\lambda$}
\DeclareUnicodeCharacter{0393}{$\Gamma$}
\DeclareUnicodeCharacter{211D}{$\mathbb{R}$}
\DeclareUnicodeCharacter{2212}{-} 

\newcommand{\A}{\mathcal{A}}

\newcommand{\G}{\mathcal{G}}
\newcommand{\M}{\mathcal{M}}
\newcommand{\N}{\mathbb{N}}
\newcommand{\R}{\mathbb{R}}

\newcommand{\E}{\mathbb{E}}
\newcommand{\e}{{\sf E}}

\newcommand{\ps}{\mathcal{P}} 

\renewcommand{\d}{\mathrm{d}}
\newcommand{\D}{{\sf d}}
\newcommand{\F}{\mathcal{F}}

\renewcommand{\P}{\mathbb{P}} 
\renewcommand{\H}{\mathcal{H}}

\renewcommand{\l}{\ell} 
\newcommand{\W}{\mathcal{W}}

\newtheorem{theorem}{Theorem}[section]
\newtheorem{proposition}[theorem]{Proposition}
\newtheorem{lemma}[theorem]{Lemma}
\newtheorem{corollary}[theorem]{Corollary}

\theoremstyle{definition}

\theoremstyle{remark}
\newtheorem{rem}{Remark} 

\title{\bf Geodesic convexity and strengthened functional inequalities in submanifolds of Wasserstein space}

\author{Louis-Pierre Chaintron
and Daniel Lacker}

\date{}

\begin{document}

\maketitle

\renewcommand{\thefootnote}{} 
\footnotetext{This work was carried out during a research stay of L.P.C.\ funded by the CERMICS, École des ponts, and the DMA, ENS-PSL, and the Fondation Sciences Mathématiques de Paris. D.L.\ acknowledges support from an Alfred P.\ Sloan Fellowship and the NSF CAREER award DMS-2045328.}
\renewcommand{\thefootnote}{\arabic{footnote}} 


\begin{abstract}
We study the geodesic convexity of various energy and entropy functionals restricted to (non-geodesically convex) submanifolds of Wasserstein spaces with their induced geometry. We prove a variety of convexity results by means of a simple general principle, which holds in the metric space setting, and which crucially requires no knowledge of the structure of geodesics in the submanifold: If the EVI gradient flow of a functional exists and leaves the submanifold invariant, then the restriction of the functional to the submanifold is geodesically convex. This leads to short new proofs of several known results, such as one of Carlen and Gangbo on strong convexity of entropy on sphere-like submanifolds,
and several new results, such as the $\lambda$-convexity of entropy on the space of couplings of $\lambda$-log-concave marginals.
Along the way, we develop sufficient conditions for existence of geodesics in Wasserstein submanifolds.
Submanifold convexity results lead systematically to improvements of Talagrand and HWI inequalities which we speculate to be closely related to concentration of measure estimates for conditioned empirical measures, and we prove one rigorous result in this direction in the Carlen-Gangbo setting.  
\end{abstract}

\tableofcontents

\section{Introduction}

It is well known that Wasserstein space $(\ps_2(\R^d),\W_2)$ is a geodesic space: For any $\mu,\nu \in \ps_2(\R^d)$, there exists a curve $\rho : [0,1] \to \ps_2(\R^d)$ such that $\rho_0=\mu$, $\rho_1=\nu$, and $\W_2(\rho_t,\rho_s)= \vert t-s \vert \W_2(\rho_0,\rho_1)$ for all $t,s \in [0,1]$. What makes this fact especially useful is the concrete structure of these geodesics, also known as displacement interpolations: For $\mu$ absolutely continuous, Brenier's theorem \cite{brenier1991polar} ensures the existence of a convex function $\varphi$ such that $\nabla\varphi\#\mu=\nu$, and we may take $\rho_t=(t\nabla\varphi + (1-t)\mathsf{Id})\#\mu$. The geodesics can also be characterized as minimizers in the Benamou-Brenier formula \cite{benamou2000computational},
\begin{equation} 
\W_2^2(\mu,\nu) = \inf\bigg\{ \int_0^1 \|v_t\|_{L^2(\rho_t)}^2\,\d t : \partial_t \rho + \nabla \cdot (\rho v)=0, \ \rho_0=\mu, \ \rho_1=\nu\bigg\}. \label{intro:BenamouBrenier}
\end{equation}
Here the continuity equation, to be understood in the weak sense, is well known to characterize absolutely continuous curves in $\ps_2(\R^d)$.

These geodesics were first developed by McCann \cite{mccann1997convexity}, to prove the uniqueness of minimizers of certain natural functionals by showing that they are strictly convex along displacement interpolations.
Recall that a functional $\e : \ps_2(\R^d) \to (-\infty,\infty]$ is geodesically $\lambda$-convex if, for every geodesic $(\rho_t)_{t \in [0,1]}$, we have
\begin{equation}
\e(\rho_t) \le (1-t) \e(\rho_0) + t \e(\rho_1) - \frac{\lambda}{2} t(1-t) \W_2^2(\rho_0,\rho_1). \label{intro:geodesicconvexity}
\end{equation}
Owing to their tractability, Wasserstein geodesics have since found many other uses. For instance, geodesic (semi-)convexity is a key ingredient in constructions of gradient flows on Wasserstein space and more general metric spaces \cite{ambrosio2008gradient}. Geodesic convexity can also be used to prove a number of important functional inequalities, such as Talagrand's inequality, the logarithmic Sobolev inequality, the HWI inequality, the Brunn-Minkowski inequality, and the Pr\'ekopa-Leindler inequality.

This paper is devoted to questions of geodesic convexity of functionals restricted to Wasserstein \emph{submanifolds}. By ``submanifold'' we really just mean ``subset'' at this point, but we adopt the  more evocative geometric terminology in view of the Wasserstein space's formal Riemannian geometry introduced by Otto \cite{otto2001geometry}. Given a subset $\M$ of $\ps_2(\R^d)$ and $\mu,\nu \in \M$, the induced distance on $\M$ is defined formally by restricting the curves in \eqref{intro:BenamouBrenier} to lie in $\M$:
\begin{equation}
\D_{\M}^2(\mu,\nu) = \inf\bigg\{ \int_0^1 \|v_t\|_{L^2(\rho_t)}^2\,\d t : \partial_t \rho + \nabla \cdot (\rho v)=0, \ \rho_0=\mu, \ \rho_1=\nu, \ \rho_t \in \M \ \forall t \in [0,1]\bigg\}. \label{intro:BenamouBrenier-subspace}
\end{equation}
We emphasize that this is not the restriction of $\W_2$ to $\M\times \M$, but rather the induced length metric. A minimizer, if one exists, is a $\M$-geodesic from $\mu$ to $\nu$ after reparametrization by arc length.
For a given energy functional $\e$ on $\M$, it is natural to study its geodesic convexity properties in the induced geometry: We say that $\e$ is geodesically $\lambda$-convex in $\M$ if \eqref{intro:geodesicconvexity} holds for all $\M$-geodesics $(\rho_t)_{t \in [0,1]}$, with $\W_2$ replaced by $\D_{\M}$ therein.

The main objectives of this article are the following:
\begin{itemize}
\item To develop what we call the \emph{EVI invariance principle}, a simple but powerful method for proving geodesic convexity of a functional in the induced geometry of a submanifold.
See Section \ref{se:metric} for the general principle in a metric space setting and Section \ref{se:examples} for applications in Wasserstein space.
\item To establish general criteria for the existence of geodesics in submanifolds (Section \ref{se:geodesics}).
\item To initiate the study of the functional inequalities in submanifolds which result from geodesic convexity (Sections \ref{se:talagrand-metric} and \ref{se:HWI-metric}), as well as their consequences for concentration of measure (Section \ref{se:transport-LD}).
\end{itemize}
Questions of convexity in submanifolds have arisen naturally in a variety of prior work studying energy-minimization problems on Wasserstein submanifolds and related constructions of constrained gradient flows. Early work of this nature by  Carlen-Gangbo \cite{carlen2003constrained,carlen2004solution} was motivated by the construction of dynamics with conserved quantities. A major recent impetus comes from the use of gradient flow methods in sampling \cite{chewi-book}. More specifically, in variational inference in Bayesian statistics \cite{wainwright-jordan}, a class of ``computationally tractable'' measures $\M$ is given, along with a typically complicated (posterior) distribution $\pi$. The minimizer of the relative entropy $H(\cdot\,|\,\pi)$ over $\M$ then serves as a more tractable surrogate for $\M$. Common choices of $\M$ are the set of Gaussian measures and the set of product measures (a.k.a.\ the mean field class). The papers \cite{lambert2022variational} and \cite{lacker2023independent} recently (respectively) identified and analyzed the associated gradient flows for $H(\cdot\,|\,\pi)$ in these two induced submanifolds of Wasserstein space.

The submanifolds of Gaussians and product measures just described are \emph{geodesically convex}, in the sense that for any two points in the submanifold there exists a geodesic in $\ps_2(\R^d)$ connecting the two points which is entirely contained in the submanifold. Trivially, the induced distance $\D_{\M}$ then coincides with $\W_2$, and the restriction of a geodesically convex function on $\ps_2(\R^d)$ to a geodesically convex submanifold remains geodesically convex in the induced submanifold geometry. With the general theory of gradient flows on metric spaces \cite{ambrosio2008gradient} in mind, we should then expect the same convergence rate for the gradient flow constrained to the submanifold, as well as similar functional inequalities relating the distance function, energy functional, and its (metric) gradient.

Non-geodesically convex submanifolds pose significantly greater difficulties, as the induced geometry and geodesics are genuinely distinct. We know of four examples which have appeared in prior work, each of which will be discussed in more detail later:
\begin{itemize}
    \item The early paper of Carlen-Gangbo \cite{carlen2003constrained} analyzes a sphere-like submanifold of measures with a constrained mean and variance (trace of covariance matrix), which is the only example we know of in which geodesics can be computed explicitly.
    \item A related example in which the entire covariance matrix is constrained was studied recently in \cite{burger2025covariance}, motivated by connections with ensemble Kalman methods and inverse problems.
    \item The paper \cite{conforti2023projected} studied gradient flows in submanifolds of couplings, with two (or more) marginals prescribed, motivated by connections with entropic optimal transport.
    \item Brenier's models of \emph{incompressible optimal transport} \cite{brenier1989least,brenier2020examples} involve the study of geodesics in a submanifold of joint probability measures with one marginal constrained. This arises as a relaxation of Arnold's interpretation of the Euler equations as a geodesic equation in the space of measure-preserving diffeomorphisms \cite{arnold1966geometrie}.
\end{itemize}
In each of these examples the submanifold takes the form $\M=\F^\perp$ for some set of continuous functions $\F$, where we define
\begin{equation}
\F^\perp = \bigg\{\rho \in \ps_2(\R^d) : \F \subset L^1(\rho), \ \int f\,\d\rho =0 \ \ \forall f \in \F\bigg\}. \label{intro:Gperp}
\end{equation}
Despite geodesics in Wasserstein space being well-understood, it is generally quite challenging to prove much about geodesics in a non-convex submanifolds such as $\F^\perp$. Geodesics can be formally described in terms of PDEs, with the constraints $\F$ contributing a Lagrange multiplier-type term to the PDE. These PDEs typically cannot be solved explicitly, and even worse, proving regularity of the multipliers appears to be an extremely difficult problem. Indeed, in the example of incompressible optimal transport, the multiplier is the pressure field, and significant technical efforts were required in \cite{ambrosio2008regularity} to prove that it is not just a distribution but in fact a function. Still more regularity would be required in order to justify the calculations one would naturally encounter in trying to prove geodesic convexity properties. We discuss these regularity problems in more detail in Section \ref{se:IncompressibleOT}.

\subsection{A general convexity principle} \label{ssec:ConvPrinc}

In this paper we propose a simple general principle which will allow us to prove a variety of convexity properties for functionals on Wasserstein submanifolds. A crucial feature of this approach is that it requires no knowledge whatsoever of the structure or regularity of geodesics. 
We give an informal statement here, with the precise metric space version stated in Theorem \ref{th:main} below. It involves the notion of \emph{evolution variational inequality}, or \emph{$\mathrm{EVI}_\lambda$ gradient flow}, which will be defined in Section \ref{se:metric}.

\begin{quote}
    \textbf{EVI invariance principle} (informal): Let $\lambda \in \R$, $\M \subset \ps_2(\R^d)$, and $\e : \ps_2(\R^d) \to (-\infty,\infty]$. Suppose the $\mathrm{EVI}_\lambda$ gradient flow of $\e$ exists and leaves $\M$ invariant; that is, if initialized in $\M$, then the gradient flow remains in $\M$ at all times. Then the restriction of $\e$ is geodesically $\lambda$-convex in $\M$ (in the induced geometry).
\end{quote}

A sufficient condition for the existence of an $\mathrm{EVI}_\lambda$ flow is that $\e$ is proper, lower semicontinuous, coercive, and $\lambda$-convex along generalized geodesics \cite[Theorem 11.2.1]{ambrosio2008gradient}. This includes the most famous examples of potential, interaction, and internal energy functionals \cite[Theorem 11.2.8]{ambrosio2008gradient}. As a partial converse, the existence of an $\mathrm{EVI}_\lambda$ flow implies that $\e$ is $\lambda$-convex along geodesics; this is due to \cite[Theorem 3.2]{daneri2008eulerian} and plays a crucial role the proof of the EVI invariance principle. The other key ingredient in the proof is a result of \cite[Theorem 3.5]{ambrosio2016optimal} showing that an $\mathrm{EVI}_\lambda$-flow in a given metric remains an $\mathrm{EVI}_\lambda$-flow in the induced length metric.
We give an alternative proof in Appendix \ref{se:alternateproof} based on the methods of \cite{baradat2020small,monsaingeon2023dynamical}.
In fact, the proof of the EVI invariance principle is quite short given prior results in the literature. Nonetheless, it appears to be a powerful tool in proving geodesic convexity properties on Wasserstein submanifolds. 

Let us illustrate this with a few examples.

\subsection{Carlen-Gangbo spheres} \label{intro:se:CarlenGangbo}
The first detailed analysis of a submanifold of $\ps_2(\R^d)$ was in \cite{carlen2003constrained}, which studied \emph{sphere-like} submanifolds
\[
\mathcal{S}_{u,\theta} := \bigg\{\rho \in \ps_2(\R^d) : \int x\,\rho(\d x) = u, \ \int |x-u|^2\,\rho(\d x) = d\theta\bigg\}, \ \ u \in \R^d, \ \theta > 0,
\]
which we will henceforth refer to as \emph{Carlen-Gangbo spheres}. Remarkably, they were able to compute explicit formulas for both the geodesics and the induced distance $\D_{\mathcal{S}_{u,\theta}}$.
Using these explicit formulas, they proved the striking fact that the differential entropy functional $H(\rho)=\int\rho\log\rho$ is geodesically $1/\theta$-convex in $\mathcal{S}_{u,\theta}$. Note that $H$ is geodesically convex but not strictly convex on the whole space $\ps_2(\R^d)$, as it is invariant under translations.

The EVI invariance leads to a short new proof, which goes as follows (with some details deferred to Theorem \ref{th:CGentropy}): Let $N_{u,\Sigma}$ denote the Gaussian measure of mean $u$ and covariance matrix $\Sigma$, for any positive definite matrix $\Sigma$. Define $\e : \ps_2(\R^d) \to [0,\infty]$ as the relative entropy
\[
\e(\rho) := H(\rho\,|\,N_{u,\theta I}) := H(\rho) + \frac{1}{\theta}\int |x-u|^2\,\rho(dx) + \frac{d}{2}\log(2\pi\theta).
\]
The terms after $H(\rho)$ are constant functions of $\rho \in {\mathcal{S}_{u,\theta}}$, and thus the geodesic convexity of $H$ in $\mathcal{S}_{u,\theta}$ is equivalent to that of $\e$.
The functional $\e$ is well known to be geodesically $1/\theta$-convex on $\ps_2(\R^d)$ due to the strong log-concavity of $N_{u,\theta I}$, and it generates an $\mathrm{EVI}_{1/\theta}$-gradient flow \cite[Theorem 11.2.8]{ambrosio2008gradient} which is governed by the Fokker-Planck equation associated with the Ornstein-Uhlenbeck process,
\begin{equation}
\partial_t\rho_t(x) = \mathrm{div}\Big(\frac{1}{\theta}(x-u) \rho_t(x)\Big) + \Delta \rho_t(x). \label{intro:FKP}
\end{equation}
Given a solution $(\rho_t)_{t \ge 0}$  of \eqref{intro:FKP}, define
\[
f(t) =  \int (x-u)\,\rho_t(x)\, \d x, \qquad g(t) = \int |x-u|^2\,\rho_t(x)\, \d x - d\theta.
\]
A simple calculation using \eqref{intro:FKP} shows that $f'(t)=-f(t)/\theta$ and $g'(t)=-2 g(t)/\theta$.
If $\rho_0 \in \mathcal{S}_{u,\theta}$, then $f(0)=0$ and $g(0)=0$, and so $f$ and $g$ are identically zero. Thus $\rho_t \in {\mathcal{S}_{u,\theta}}$ for all $t > 0$, and we conclude that the $\mathrm{EVI}_\lambda$-flow leaves $\mathcal{S}_{u,\theta}$ invariant.
The geodesic $1/\theta$-convexity of $\e$ in $\mathcal{S}_{u,\theta}$, and thus that of $H$, now follows from our EVI invariance principle.

\subsection{Additional examples}

Let us summarize the other main examples that will be developed in detail in Section \ref{se:examples}:

\begin{enumerate}[(A)]
\item Theorem \ref{th:CGpolynomials} shows that $H$ is convex on several variants of the Carlen-Gangbo sphere. For example, consider a symmetric positive definite matrix $\Sigma$, and define
\[
\mathcal{S}_{u,\Sigma} := \bigg\{\rho \in \ps_2(\R^d) : \int x\,\rho(\d x) = u, \ \int (x-u)^{\otimes 2}\,\rho(\d x) = \Sigma\bigg\}.
\]
If $\Sigma \ge \theta I$, then $H$ is again geodesically $1/\theta$-convex in $\mathcal{S}_{u,\theta}$. This was shown by Burger et al.\ \cite[Theorem 1.23]{burger2025covariance} for $(u,\Sigma)=(0,I)$. In addition, we may add the constraints $\int p(x)\,\rho(\d x)=\int p(x)\,N_{u,\Sigma}(\d x)$, for polynomials  $p$ up to a given degree, and the same convexity result holds. 
\item Theorem \ref{th:logenergy} shows that the logarithmic energy $\ps_2(\R) \ni \rho \mapsto -\int_\R\int_\R\log|x-y|\rho(\d x)\rho(\d y)$ is geodesically $1/\theta$-convex in ${\mathcal{S}_{u,\theta}}$.
\item There is a more general principle behind the scenes in the example of (A) above. Suppose $V$ is a $\lambda$-convex function such that $\pi(\d x)=e^{-V(x)}\d x$ is a probability measure, and let $\F$ be a set of eigenfunctions of the operator $\Delta-\nabla V \cdot \nabla$. Then the relative entropy functional $H(\cdot\,|\,\pi)$ is geodesically $\lambda$-convex when restricted to  $\F^\perp$, which was defined in \eqref{intro:Gperp}. 
This extends to Riemannian manifolds in place of $\R^d$. 
Note when $\pi$ is Gaussian that the associated eigenfunctions are the Hermite polynomials (or transformations thereof, if $(u,\Sigma)\neq(0,I)$), and we recover (A).
\item Suppose $\mu_1,\ldots,\mu_n\in\ps_2(\R^d)$ are given probability densities such that $-\log \mu_i$ is $\lambda$-convex for each $i$, for some $\lambda \in \R$. Let $\Pi=\Pi(\mu_1,\ldots,\mu_n)$ denote the set of couplings, i.e., the set of probability measures on $\R^{nd}$ with marginals $\mu_1,\ldots,\mu_n$. It was shown in \cite[Proposition 2.2]{conforti2023projected} that $\Pi$ is not geodesically convex unless it is a singleton. We show in Theorem \ref{thm:CouplingHConvex} that the differential entropy $H$ is geodesically $\lambda$-convex in $\Pi$.
\item In the setting of (the multiphase formulation of) incompressible optimal transport (IOT), Lavenant \cite{lavenant2017time} and Baradat-Monsaingeon \cite{baradat2020small} recently proved a conjecture of Brenier \cite[Section 4]{brenier2003extended} by showing that the averaged entropy is geodesically convex in $\Pi( \mathrm{Leb}, \mathrm{Leb} )$ equipped with the IOT metric. 
See Section \ref{se:IncompressibleOT} for a detailed presentation of this setting, including a discussion of the regularity issues that arise in the study of geodesics, as well as a simple new proof of Brenier's conjecture using the EVI invariance principle.
\end{enumerate}

\subsection{Existence of geodesics}

The EVI invariance principle is vacuous unless we know that the submanifold $\M$ admits geodesics. In Section \ref{se:geodesics} we give two theorems providing sufficient conditions for this.
The second and more novel of these theorems, tailored to the setting of \eqref{intro:Gperp}, is based on an adaptation of a Fenchel-Rockafellar duality argument of Brenier \cite{brenier2020examples}, which yields both existence of minimizers and a duality formula for $\D_\M$.
It applies only after proving a priori that $\D_\M$ is finite, i.e., that there exist curves in $\M$ of finite length between any two points, which is inobvious in general. However, in the setting of the EVI invariance principle with $\lambda > 0$, there is a natural candidate: Run the gradient flow starting from each point, noting that each converges  to the same long-time limit, and concatenate one of these curves with the time-reversal of the other to produce a curve connecting the two initial points.
Showing that this results in a curve of finite length is somewhat subtle and relies on the short-time regularization of the gradient flow.
See Section \ref{se:finitecurves} for details.
One notable consequence of our analysis is a proof (in Theorem \ref{th:CGpolynomials}) of existence of geodesics in the covariance-constrained submanifold $\mathcal{S}_{u,\Sigma}$, answering an open question from  \cite{burger2025covariance}.

\subsection{Strengthened functional inequalities} \label{se:IntroFunctionIneq}

The geodesic convexity of functionals on Wasserstein space has been used to prove several important functional inequalities.
A first example is Talagrand's inequality. Suppose $\pi$ is a  $\lambda$-log-concave probability density, in the sense that $-\log\mu$ is a $\lambda$-convex function on $\R^d$. Then the relative entropy $H(\cdot\,|\,\pi)$ is geodesically $\lambda$-convex in $\ps_2(\R^d)$. If $\lambda > 0$, this geodesic convexity leads to a quick proof of the \emph{Talagrand inequality},
\begin{equation}
\W_2^2(\mu,\pi) \le \frac{2}{\lambda}H(\mu\,|\,\pi), \quad \forall \mu \in \ps_2(\R^d). \label{intro:Talagrand}
\end{equation}
In fact, as was shown in \cite{otto2000generalization}, this follows from the stronger \emph{HWI inequality},
\begin{equation}
H( \mu\,|\,\pi ) - H(\nu\,|\,\pi) \leq \W_2( \mu,\nu) I(\mu\,|\,\pi) - \frac{\lambda}{2} \W_2^2 (\mu,\nu), \quad \forall \mu,\nu \in \ps_2(\R^d), \ H(\nu\,|\,\pi) < \infty.
\end{equation}
If $\e=H(\cdot\,|\,\pi)$ and $\M \subset \ps_2(\R^d)$ satisfy the conditions of the EVI invariance principle (with $\lambda > 0$), then we show in Propositions \ref{pr:MetricTalagrand} and \ref{pr:HWI} that analogous inequalities hold on $\M$, with $\W_2$ replaced by the induced distance:
\begin{align}
\D_{\M}^2(\mu,\pi) &\le \frac{2}{\lambda}H(\mu\,|\,\pi), \quad \forall \mu \in \M,  \label{intro:Talagrand-strengthened} \\
H( \mu\,|\,\pi ) - H(\nu\,|\,\pi) &\le \D_{\M}( \mu,\nu) I(\mu\,|\,\pi) - \frac{\lambda}{2} \D_{\M}^2 (\mu,\nu), \quad \forall \mu,\nu \in \M, \ H(\nu\,|\,\pi) < \infty.
\end{align}
For measures in $\M$ these inequalities are stronger, because $\D_{\M} \ge \W_2$. Here $I$ is the relative Fisher information, which corresponds to the metric gradient of $H(\cdot\,|\,\pi)$, and an interesting detail is that this metric gradient is the same in both $\ps_2(\R^d)$ and $\M$ in the context of the EVI invariance principle; see Proposition \ref{pr:slopes}.

These inequalities could be compared to  \emph{stability} results for functional inequalities. Stability results, such as for Talagrand's inequality as studied in \cite{mikulincer2021stability,bez2023stability}, typically take the form of lower bounds on the deficit $(2/\lambda)H(\mu\,|\,\pi) - \W_2^2(\mu,\pi)$ in terms of quantities measuring how far $\mu$ is to the (set of) equality cases.

The Talagrand inequality is well known to characterize various properties of $\pi$ pertaining to concentration of measure and large deviations \cite{gozlan2010transport}. 
The entropy $H(\cdot\,|\,\pi)$ famously arises in Sanov's theorem as the large deviation rate function for the empirical measure of iid samples from $\pi$. The restriction of a rate function to a subset is suggestive of the Gibbs conditioning principle (and this link makes the inequality \eqref{intro:Talagrand-strengthened} rather distinct from the other stability results in the literature mentioned above).
This leads us to speculate that the strengthened inequality \eqref{intro:Talagrand-strengthened} on $\M$ encodes concentration and large deviation properties for a \emph{conditioned} empirical measure.

We report a one such result in Section \ref{se:TalagrandLD} in the case of the Carlen-Gangbo sphere $\mathcal{S}_{0,1}$, though we believe there is much more to explore in this direction. Let us first quote a remarkable explicit formula from \cite[Equation (3.30)]{carlen2003constrained}:\footnote{Note that in \cite{carlen2003constrained} the definition of $\W_2^2$ is half of the usual definition, and we adopt the latter; hence our expression for $\D_{\mathcal{S}_{0,1}}^2$ is twice that of \cite{carlen2003constrained}.}
\[
\D_{\mathcal{S}_{0,1}}(\mu,\nu) = 2\sqrt{d}\arcsin\bigg(\frac{\W_2(\mu,\nu)}{2\sqrt{d}}\bigg).
\]
The corresponding strengthened Talagrand inequality \eqref{intro:Talagrand-metric} was in fact already observed in \cite{carlen2003constrained}. Defining a function $\alpha : \R \to [0,\infty]$ by
\[
\alpha(x) := \begin{cases}
2d\arcsin^2(\sqrt{x}/2\sqrt{d}) &\text{if } 0 \le x \le 2\sqrt{d} \\ \infty &\text{otherwise},
\end{cases}
\]
we can write the inequality \eqref{intro:Talagrand-metric} for $\M=\mathcal{S}_{0,1}$ (with $\lambda=1$) as
\begin{equation}
\alpha( \W_2^2(N_{0,I},\mu)) \le J(\mu) := \begin{cases} H(\mu\,|\,N_{0,I}), &\text{if }  \mu \in \mathcal{S}_{0,1} \\
\infty &\text{if } \mu \in \ps(\R^d) \setminus \mathcal{S}_{0,1} . \end{cases} \label{CG-Talagrand-alphaIntro}
\end{equation}
The left-hand side is weakly lower semicontinuous in $\mu$, while the right-hand side is not, so the same inequality remans valid if we replace $J$ by its lower semicontinuous envelope, which we show in Proposition \ref{pr:lscenvelope} to be
\[
\overline{J}(\mu) = \begin{cases}
H(\mu\,|\,\gamma) + \frac{d}{2}-\frac12\int |x|^2\,\mu(\d x) &\text{if }  \int x\,\mu(\d x)=0, \ \int |x|^2\,\mu(\d x) \le d \\ \infty &\text{otherwise}. \end{cases}
\]
This is precisely the rate function governing the large deviations principle for the centered empirical measure
\begin{equation}
\frac{1}{n}\sum_{i=1}^n \delta_{U^n_i - \frac{1}{n}\sum_{j=1}^nU^n_j}, \label{intro:empiricalmeasure-unif}
\end{equation}
where $(U^n_1,\ldots,U^n_n)$ is a random vector in $\R^{nd}$ uniformly distributed on the sphere $\sqrt{nd}\mathbb{S}^{nd-1}$, with $U^n_i$ denoting the $i$th block of $d$ coordinates. By analogy, the right-hand side $H(\mu\,|\,N_{0,I})$ of the usual Talagrand inequality is the rate function of Sanov's theorem for the empirical measure of iid Gaussians.
Exploiting this connection, we explain in Proposition \ref{pr:LDP-transport} a characterization of the inequality \eqref{CG-Talagrand-alphaIntro} in terms of a large deviation inequality for the empirical measure \eqref{intro:empiricalmeasure-unif}, inspired by the work of Gozlan-L\'eonard on connections between transport inequalities and large deviations \cite{gozlan2007large}.

In Section \ref{se:BianeVoiculescu}, we discuss in less detail an inequality on $\mathcal{S}_{0,1}$ which arises from our result on the uniform convexity of the logarithmic energy functional. Proposition \ref{pr:BianeVoiculescu} strengthens (on $\mathcal{S}_{0,1}$) an inequality due to Biane-Voiculescu which can be seen as the analogue of Talagrand's inequality in free probability. We discuss a similar but still conjectural connection with the large deviations of the empirical measure of the eigenvalues of a fixed trace ensemble of random matrices.

\subsection{Outline}
The rest of the paper is organized as follows. In Section \ref{se:metric} we develop the EVI invariance principle in full detail, in the general setting of metric spaces, along with metric space forms of the strengthened Talagrand and HWI inequalities. In Section \ref{se:geodesics} we provide sufficient conditions for the existence of geodesics in submanifolds. Section \ref{se:examples} then develops the examples described above, and finally Section \ref{se:transport-LD} discusses strengthened functional inequalities in connection with concentration and large deviations.

\section{Generalities in metric spaces} \label{se:metric}

We begin by recalling some metric space concepts, borrowing most of our definitions from \cite{ambrosio2016optimal}.
Let $(X,\D_X)$ be an extended metric space, which means that $\D_X$ satisfies all of the axioms of a metric except is allowed to take the value $+\infty$. We allow this generality because we will encounter it inevitably below when defining induced length distances on subspaces.
An \emph{$X$-geodesic} is defined as a curve $\gamma :[0,1] \to X$ satisfying 
\[
\D_X(\gamma_s,\gamma_t) = |t-s|\D_X(\gamma_0,\gamma_1) < \infty, \quad \forall s,t \in [0,1].
\]
For $\lambda \in \R$, a  function $\e : X \to (-\infty,\infty]$ is said to be \emph{$\lambda$-convex along a curve} $\gamma : [0,1] \to X$ if 
\[
 \e(\gamma_t) + \frac{\lambda}{2} t(1-t) \D_X^2(\gamma_0,\gamma_1) \le (1-t) \e(\gamma_0) + t \e(\gamma_1) , \quad \forall t \in [0,1].
\]
The function $\e$ is called \emph{geodesically $\lambda$-convex in $X$} if for any $X$-geodesic  $\gamma$ it holds that $\e$ is $\lambda$-convex along $\gamma$.
If $\lambda=0$ we simply say \emph{geodesically convex in $X$}. Note that at this stage we have made no assumptions about existence of geodesics, and if none exist then the concept of a geodesically convex function on $X$ is vacuous. In Section \ref{se:geodesics} we will discuss additional assumptions which ensure the existence of geodesics.

For a bounded interval $I \subset \R$ and for $p \in [1,\infty]$ we let $AC^p (I;(X,\D_X))$ denote the set of maps $\gamma : I \to X$ for which there exists $m \in L^p(I)$ satisfying
\[
\D_X(\gamma_s,\gamma_t) \le \int_s^t m(u)\,\d u, \quad \forall s,t \in I, \ s < t.
\]
When $p=1$ we write $AC$ instead of $AC^1$.
For $\gamma \in AC(I;(X,\D_X))$,
define the metric derivative $|\dot\gamma| : I \to [0,\infty]$ as
\begin{equation} \label{eq:MetricSlope}
|\dot\gamma|_t = \limsup_{s\to t}\frac{\D_X(\gamma_s,\gamma_t)}{|s-t|},
\end{equation}
which exists for a.e.\ $t \in I$ and defines a function in $L^p(I)$. We have $\D_X(\gamma_s,\gamma_t)\le\int_s^t|\dot\gamma|_u\,\d u$ for all $t,s \in I$ with $t > s$.
The space $AC^p_{\mathrm{loc}}((0,\infty);(X,\D_X))$ consists of maps  $\gamma : (0,\infty) \to X$ whose restriction to each bounded interval $I$ is in $AC^p(I;(X,\D_X))$.
Note for $\gamma \in AC^p_{\mathrm{loc}}((0,\infty);(X,\D_X))$ that any two points in the image $\mathrm{Im}(\gamma):=\{\gamma_t : t \in (0,\infty)\}$ must be at finite distance.

We next define the crucial concept of \emph{evolution variational inequality}, which has risen to prominence as a particularly powerful notion of gradient flow in a metric space; we refer to  the lecture notes \cite{daneri2010lecture} and the more recent paper \cite{muratori2020gradient} for overviews. We adopt \cite[Definition 3.2]{ambrosio2016optimal}. Let $\e : X \to (-\infty,\infty]$ have nonempty domain $D(\e) = \{\e < \infty\}$. Let $\lambda \in \R$. We say $\gamma \in AC^2_{\mathrm{loc}}((0,\infty);(X,\D_X))$ is an $\mathrm{EVI}_\lambda$-flow of $\e$ in $(X,\D_X)$ if $t \mapsto \e(\gamma_t)$ is lower semicontinuous and finite-valued, and for all $x \in D(\e)$ such that $\D_X(y,\mathrm{Im}(\gamma)) < \infty$ we have 
\begin{equation} \label{eq:EVI}
 \frac{1}{2} \frac{\d^+}{\d t} \D_X^2 (\gamma_t,x) + \frac{\lambda}{2} \D_X^2 ( \gamma_t , x )  \leq \e(x) - \e(\gamma_t), \ \ \forall t > 0,
\end{equation} 
where $\d^+/\d t$ denotes the upper right-derivative.
We say that $\gamma$ starts at $x \in D(\e)$ if $\D_X(\gamma_t,x)\to 0$ as $t \downarrow 0$ and $\liminf_{t\downarrow 0}\e(\gamma_t) \ge \e(x)$ (the latter being automatic if $\e$ is lower semicontinuous); this is a somewhat stronger definition of ``starts at'' compared to \cite[Definition 3.2]{ambrosio2016optimal}, but will suffice for our purposes.

An EVI$_\lambda$-flow $\gamma$ has a number of important properties, stated in developed in detail in \cite[Theorem 3.5]{muratori2020gradient}. We will most frequently use the \emph{energy identity}\footnote{Note that \cite[(3.17)]{muratori2020gradient} states the energy identity not for $|\dot\gamma^i|_t$ but rather for the right-limit $\lim_{h\to 0+}\D_X(\gamma^i_{t+h},\gamma^i_t)/h$. This agrees a.e.\ with $|\dot\gamma^i|_t^2$ because, unlike \cite[Theorem 3.5]{muratori2020gradient}, our definition of EVI flow builds in the assumption that $\gamma^i$ is absolutely continuous.}
\begin{equation} \label{eq:energyidentity}
\frac{\d}{\d t} \e(\gamma_t) = -|\dot\gamma|_t^2=|\partial\e|^2(\gamma_t), \quad a.e. \ t > 0.
\end{equation}
Here the slope $\vert \partial \e \vert : X \to [0,\infty]$ of  $\e$ at a point $x \in X$ is set to $+\infty$ if $\e(x)=\infty$, to $0$ if $x$ is isolated, and otherwise to
\[
\vert \partial \e \vert(x) := \limsup_{y \rightarrow x} \frac{(\e(x) − \e(y))_+}{\D_X(x,y)}.
\]

The main examples in this paper involve $(X,\D_X)=(\ps_2(\R^d),\W_2)$ being the Wasserstein space over $\R^d$, with the functional $\e$ being a combination of classical energy functionals.  
The differential entropy is defined by $H(\rho) = \int_{\R^d} \rho\log\rho$ when $\rho$ admits a density and $H(\rho)=\infty$ otherwise; note that the negative part of $\rho\log\rho$ is integrable for any density in $\ps_2(\R^d)$, so $H$ is well defined with values in $(-\infty,\infty]$. For $\lambda\in\R$ and a $\lambda$-convex function $V : \R^d \to \R$, we also work with the potential energy functional $\rho \mapsto \int V\,\d\rho$. The functional
\[
\e(\rho) = H(\rho) + \int V\,\d\rho
\]
is geodesically $\lambda$-convex in $\ps_2(\R^d)$, and the $\mathrm{EVI}_\lambda$-flow for $\e$ exists starting from every point in $\ps_2(\R^d)$. Moreover, the $\mathrm{EVI}_\lambda$-flow $(\rho_t)_{t \ge 0}$ satisfies the Fokker-Planck equation
\[
\partial_t \rho = \nabla \cdot (\rho V) + \Delta \rho,
\]
in the sense of distributions. This is shown in \cite[Theorem 11.2.8]{ambrosio2008gradient}.

\subsection{The induced length metric}

A subset $Y$ of $X$ will always be equipped with the length metric  (or inner metric) induced by $\D_X$. That is, 
\begin{equation} \label{eq:DefDY}
\D_Y(y,y') = \inf\bigg\{ \int_0^1|\dot\gamma|_t\,\d t : \gamma \in AC([0,1];(Y,\D_X|-{Y\times y})), \, \gamma_0=y, \, \gamma_1=y'\bigg\}.
\end{equation}
Instead of writing $\gamma \in AC([0,1];(Y,\D_X|-{Y\times y}))$ we could equivalently write ``$\gamma \in AC([0,1];(X,\D_X))$ and $\gamma_t \in Y$ for all $t \in [0,1]$.''
Note that this may take the value $+\infty$ even if $\D_X$ does not. We always have $\D_Y \ge \D_X$ on $Y \times Y$.
We will say that \emph{$Y$ is a geodesic space} if for every $y,y'\in Y$ there exists a $Y$-geodesic between them, i.e., a map $\gamma : [0,1] \to Y$ satisfying $\gamma_0=y$,  $\gamma_1=y'$, and $\D_Y(\gamma_s,\gamma_t)=|t-s|\D_Y(y,y') < \infty$ for all $t,s \in [0,1]$.

The distance $\D_Y$ admits several equivalent forms. We have
\begin{equation} \label{eq:DefDY-squared}
\D_Y(y,y') = \inf\bigg\{ \bigg( \int_0^1|\dot\gamma|_t^2\,\d t \bigg)^{1/2}: \gamma \in AC^2([0,1];(Y,\D_X|_{Y\times Y})), \, \gamma_0=y, \, \gamma_1=y'\bigg\}.
\end{equation}
Indeed, this is because any curve $\gamma \in AC([0,1];(Y,\D_X|_{Y\times Y}))$ can be reparametrized by arc length, i.e., to have a.e.\ constant metric derivative $|\dot\gamma|$ (see, e.g., \cite[Lemma 1.1.4(b)]{ambrosio2008gradient} or \cite[Proposition 2.5.9]{burago2001course}).
Alternatively, recall that the \emph{length} of $\gamma : [0,1] \to (X,\D_X)$ is 
\begin{equation} \label{eq:discLength}
\l ( \gamma ) := \sup \bigg\{ \sum_{k=0}^{n-1} \D_X ( \gamma_{t_k}, \gamma_{t_{k+1}} ) : 0 = t_0 < t_1 < \ldots < t_n = 1  \bigg\} 
\end{equation}
is finite.
According to \cite[Theorem 2.7.6]{burago2001course}, the metric derivative $|\dot\gamma|$ given by \eqref{eq:MetricSlope} of a Lipschitz curve then satisfies 
\begin{equation} \label{eq:lenghtCont}
\l ( \gamma ) = \int_0^1 \vert \dot\gamma \vert_t \, \d t.
\end{equation} 
Hence,
\begin{equation} \label{eq:DefDY-length}
\D_Y(y,y') = \inf\bigg\{\ell(\gamma) : \gamma : [0,1] \to (Y,\D_X|_{Y\times Y}) \ \text{Lipschitz}, \ \gamma_0=y, \, \gamma_1=y'\bigg\}.
\end{equation}
We will not use this last formulation.

\subsection{The EVI invariance principle}

Our \emph{EVI invariance principle} described in the introduction takes the following general form. It follows so quickly from prior results that it seems scarcely worthy of a designation as a ``theorem,'' but we will illustrate its power in  many examples in Section \ref{se:examples}, where we believe it will earn the title.

\begin{theorem} \label{th:main}
Let $\lambda \in \R$.
Let $(X,\D_X)$ be an extended metric space, and let $\e : X \to (-\infty,\infty]$ be bounded from below with nonempty domain $D(\e)=\{\e < \infty\}$.
Let $Y \subset X$.
Suppose that for each $y \in Y \cap D(\e)$ there exists an $\mathrm{EVI}_\lambda$-flow $(\gamma^y_t)_{t>0}$ of $\e$ in $(X,\D_X)$ which starts from $y$, such that $\gamma^y_t \in Y$ for each $t > 0$. Then $(\gamma^y_t)_{t>0}$ is also an $\mathrm{EVI}_\lambda$-flow of $\e$ in $(Y,\D_Y)$ starting from $y$, for each $y \in Y \cap D(\e)$. Moreover, $\e$ is geodesically $\lambda$-convex in $Y$.
\end{theorem}
\begin{proof}
First observe that the $\mathrm{EVI}_\lambda$ property is trivially inherited in the restriction metric, so that $(\gamma^y_t)_{t>0}$ is in fact an $\mathrm{EVI}_\lambda$-flow of $\e$ in $(Y,\D_X|_{Y \times Y})$. It was shown in \cite[Theorem 3.5]{ambrosio2016optimal} that an $\mathrm{EVI}_\lambda$-flow in a given metric remains an $\mathrm{EVI}_\lambda$-flow in the induced length metric, so  that $(\gamma^y_t)_{t>0}$ is an $\mathrm{EVI}_\lambda$-flow of $\e$ in $(Y,\D_Y)$ starting from $y$.
We then argue as in \cite[Theorem 2.9]{daneri2010lecture}, which shows that the existence of an $\mathrm{EVI}_\lambda$-flow from every starting point $y \in Y \cap D(\e)$ implies the $\lambda$-convexity of $\e$ along $Y$-geodesics whose endpoints are in $D(\e)$. (This argument does not require finiteness of the metric.) The $\lambda$-convexity along $\gamma$ is trivial if at least one endpoint is not in $D(\e)$, and we deduce that $\e$ is $\lambda$-convex along any $Y$-geodesic.

There is some subtlety above in saying that $(\gamma^y_t)_{t>0}$ is an $\mathrm{EVI}_\lambda$-flow of $\e$ in $(Y,\D_Y)$ \emph{starting from $y$}. One should check that $\D_Y(\gamma^y_t,y) \to 0$, which of course does not follow from $\D_X(\gamma^y_t,y)\to 0$. However, the last inequality in the proof of \cite[Theorem 3.5]{ambrosio2016optimal} applied with $y=x$ shows that
\[
\frac12 \D_Y^2(\gamma^y_t,y) \le t\big(\e(y)-\e(\gamma^y_t)\big).
\]
From the assumption that $\e$ is bounded from below, and because $y \in D(\e)$, we deduce that $\D_Y(\gamma^y_t,y) \to 0$.
\end{proof}

We give an alternate proof in Appendix \ref{se:alternateproof}, based on a variational technique developed by Baradat-Monsaingeon \cite{baradat2020small} and Monsaingeon et al \cite{monsaingeon2023dynamical}. 
These papers share a similar philosophy to ours, in the sense that they prove geodesic convexity in some specific models using methods that do not require detailed knowledge of geodesics.

In Theorem \ref{th:main}, the statement that $\e$ is geodesically $\lambda$-convex in $Y$ might be vacuous, if $Y$ does not contain many geodesics.\footnote{Even if $Y$ is not a geodesic space, we could still make statements about \emph{approximate} geodesic convexity, as explained in \cite[Theorem 3.2]{daneri2008eulerian}.} We will return to this point in Section \ref{se:geodesics}, where we provide sufficient conditions for $Y$ to be a geodesic space. First, we discuss the general principles behind the strengthened forms of functional inequalities introduced in Section \ref{se:IntroFunctionIneq}.

\subsection{Talagrand inequalities} \label{se:talagrand-metric}

The $\lambda$-geodesic convexity of a functional can be used to prove various functional inequalities, and we explain some here in the context of the EVI invariance principle. The first is related to Talagrand's inequality, and it is well known already in the metric space EVI setting (e.g., \cite[(3.18a)]{muratori2020gradient}), but we include its short proof.

\begin{proposition} \label{pr:MetricTalagrand}
Suppose the assumptions of Theorem \ref{th:main} hold with $\lambda > 0$. Suppose also that $Y$ is a geodesic space and that $\inf_{y \in Y}\e(y)$ is attained at some $y_* \in Y$. Then
\begin{equation}
\D_Y^2 (y,y_*) \le \frac{2}{\lambda}(\e(y)-\e(y_*)) , \quad \forall y \in Y. \label{intro:Talagrand-metric}
\end{equation}
\end{proposition}
\begin{proof}
Let $(\gamma_t)_{t \in [0,1]}$ be a geodesic from $y_*$ to $y$. By Theorem \ref{th:main}, $\e$ is geodesically $\lambda$-convex in $Y$, and thus
\begin{align*}
0 &\le \e(\gamma_t) - \e(y_*) \le t(\e(y)-\e(y_*)) - \frac{\lambda}{2}t(1-t)\D_Y^2 (y,y_*).
\end{align*}
Divide by $t$ and send $t \downarrow 0$ to deduce the claim.
\end{proof}

The inequality \eqref{intro:Talagrand-metric} is stronger than the analogous one with $\D_X$ in place of $\D_{Y}$, simply because $\D_X \vert_{Y\times Y} \le \D_Y$. When $X$ is Wasserstein space and $\e$ is the relative entropy with respect to the standard Gaussian measure, the inequality on $X$ is Talagrand's inequality \cite{talagrand1996transportation}. The inequality \eqref{intro:Talagrand-metric} then translates to a strengthening of Talagrand's inequality when restricted to an appropriate submanifold $Y$, as was first observed in \cite{carlen2003constrained} for Carlen-Gangbo spheres. This will be discussed in more detail in Section \ref{se:TalagrandLD}.

The smaller the submanifold $Y$, the larger the left-hand side of \eqref{intro:Talagrand-metric}. The following proposition takes this observation to its extreme, by applying the EVI invariance principle when the submanifold $Y$ is taken to be the image of a single EVI curve. We give an alternative proof as well, without using the EVI invariance principle, from which Proposition \ref{pr:MetricTalagrand} could then be deduced as corollary.

\begin{proposition} \label{pr:Talagrand-extreme}
Let $(X,\D_X)$ be a complete metric space and $\lambda > 0$. Let $\e : X \to (-\infty,\infty]$ be lower semicontinuous and bounded from below with nonempty domain. Suppose $(\gamma_t)_{t > 0}$ is an $\mathrm{EVI}_\lambda$-flow of $\e$ in $(X,\D_X)$. Then, for $t > 0$,
\begin{equation}
\int_t^\infty |\dot\gamma|_s\,\d s \le \sqrt{\frac{2}{\lambda}\big(\e(\gamma_t)-\inf_{x \in X}\e(x)\big)}.
\end{equation}
\end{proposition}
\begin{proof}
First proof: By \cite[Theorem 3.5]{muratori2020gradient}, $\gamma_t$ converges as $t\to\infty$ to the unique minimizer $\gamma_\infty$ of $\e$. Define $Y = \{\gamma_t : t \in (0,\infty]\}$.
By construction, $\gamma_t \in Y$ for all $t$, and so Theorem \ref{th:main} applies, and $\e$ is geodesically $\lambda$-convex on $Y$. Moreover, $(Y,\D_Y)$ is a geodesic space, and the distance is easily seen to be given by
\begin{equation}
\D_Y(\gamma_{a},\gamma_{b}) = \int_{a}^{b}|\dot\gamma|_s \,\d s,
\end{equation}
for $0 \le a < b \le \infty$. The claim now follows from Proposition \ref{pr:MetricTalagrand}.

Second proof: We use the log-Sobolev-type inequality $\e(\gamma_t)-\e(\gamma_\infty) \le \frac{1}{2\lambda}|\partial\e|^2(\gamma_t)$, shown at this level of generality in \cite[(3.18a)]{muratori2020gradient}. Combining it with the energy identity \eqref{eq:energyidentity},
\begin{align}
\frac{\d}{\d t}\sqrt{\e(\gamma_t)-\e(\gamma_\infty)} &= -\frac{|\dot\gamma|_t^2}{2\sqrt{\e(\gamma_t)-\e(\gamma_\infty)}} \le \sqrt{\frac{\lambda}{2}} \frac{|\dot\gamma|_t^2}{|\partial\e|(\gamma_t)} = \sqrt{\frac{2}{\lambda}}|\dot\gamma|_t.
\end{align}
Integrate from $t$ to $\infty$ to complete the proof.
\end{proof}

\subsection{HWI and log-Sobolev inequalities} \label{se:HWI-metric}

Stronger than the Talagrand-type inequality of Proposition \ref{pr:MetricTalagrand} is the following metric space form of the HWI inequality.

\begin{proposition} \label{pr:HWI}
Let $\lambda \in \R$.
Let $(X,\D_X)$ be a metric space, and let $\e : X \to (-\infty,\infty]$. Suppose $Y \subset X$ is such that $(Y,\D_Y)$ is a geodesic space. Assume that $\e$ is geodesically $\lambda$-convex in $Y$, and that $\e$ satisfies the following HWI inequality in $(Y,\D_X|_{Y \times Y})$:
\begin{equation}
\e ( x )  \leq \e ( x') + \D_X ( x,x' ) \vert \partial \e  \vert( x ) - \frac{\lambda}{2} \D^2_X ( x, x' ), \quad \forall x \in Y \cap D(\vert \partial \e  \vert), \ x' \in Y.   \label{HWI-X}
\end{equation}
Then $\e$ satisfies the HWI inequality in $(Y,\D_Y)$:
\begin{equation}
\e ( y )  \leq \e ( y') + \D_Y ( y,y' ) \vert \partial \e  \vert( y ) - \frac{\lambda}{2} \D^2_Y ( y, y' ), \quad \forall y \in Y \cap  D(\vert \partial \e  \vert), \ y' \in Y.    \label{HWI-Y}
\end{equation}
\end{proposition}
\begin{proof}
Let $(\gamma_t)_{t \in [0,1]}$ be a $Y$-geodesic with $\gamma_0 \in  D(\vert \partial \e  \vert)$.
Let $t \in (0,1)$.
The HWI inequality \eqref{HWI-X} in $X$ applied to $(x,x')=(\gamma_0,\gamma_t)$ gives 
\[ 
\e ( \gamma_0 ) \le \e ( \gamma_t ) + \D_X ( \gamma_0, \gamma_t ) \vert \partial \e \vert( \gamma_0 )  - \frac{\lambda}{2} \D^2_X ( \gamma_0, \gamma_t ).
\]
Moreover, $\lambda$-convexity of $\e$ in $Y$ entails that
\[
\e(\gamma_t) \leq (1-t) \e( \gamma_0) + t \e(\gamma_1) - \frac{\lambda}{2} t(1-t) \D_Y^2 (\gamma_0,\gamma_1).
\] 
Combining these inequalities yields
\begin{equation}
t \e ( \gamma_0 )  \le t \e ( \gamma_1 ) + \D_X ( \gamma_0, \gamma_t ) \vert \partial \e \vert( \gamma_0 )  - \frac{\lambda}{2} \D^2_X ( \gamma_0, \gamma_t ) - \frac{\lambda}{2} t ( 1-t) \D^2_Y(\gamma_0,\gamma_1). \label{pf:HWI-1}
\end{equation}
We then notice that
\[ \D_X ( \gamma_0, \gamma_t ) \leq \D_Y ( \gamma_0, \gamma_t ) = t \D_Y ( \gamma_0,\gamma_1 ), \]
so that dividing by $t$ in \eqref{pf:HWI-1} and sending $t \rightarrow 0$ completes the proof.
\end{proof}

In the context of Wasserstein space, the functionals $\D_X$, $\e$, and $\vert\partial\e\vert$ are typically $\W_2$, relative entropy $H(\cdot\,|\,\pi)$ with respect to a density for which $-\log\pi$ is $\lambda$-convex, and relative Fisher information $I(\cdot\,|\,\pi)$, hence the name \emph{HWI}.
The HWI inequality in this context was discovered by Otto and Villani \cite{otto2000generalization}.
It implies both the Talagrand inequality and the log-Sobolev inequality when $\lambda > 0$. Indeed, if $\e$ is bounded from below on $X$ and $\lambda > 0$, then \eqref{HWI-X} implies 
\begin{equation}
\e(x) -  \inf_{x' \in X}\e(x') \le \frac{1}{2\lambda}\vert \partial \e \vert^2(x), \quad \forall x \in X. \label{ineq:LSI-X}
\end{equation}
The inequality \eqref{ineq:LSI-X} is also known to follow directly from the existence of EVI flows \cite[(3.18a)]{muratori2020gradient}.
In the Wasserstein space setting, the inequality \eqref{ineq:LSI-X} is precisely the log-Sobolev inequality of Bakry-\'Emery.
Although the EVI invariance principle leads to strengthened versions of the Talagrand and HWI inequalities, the log-Sobolev cannot be strengthened under the assumptions of Theorem \ref{th:main}. This is because, when the EVI-flow leaves $Y$ invariant, the metric slope of $\e|_Y$ defined intrinsically in $Y$ coincides with the restriction of $\vert\partial\e\vert$ to $Y$:

\begin{proposition} \label{pr:slopes}
Suppose the assumptions of Theorem \ref{th:main} hold. Suppose also that $(X,\D_X)$ and $(Y,\D_Y)$ are genuine (non-extended) metric spaces. Then, for all $y \in Y$,
\[
\vert \partial \e \vert(y) =  \limsup_{y' \rightarrow y, \ y' \in Y} \frac{(\e(y) − \e(y'))_+}{\D_Y(y,y')} = \limsup_{y' \rightarrow y, \ y' \in Y} \frac{(\e(y) − \e(y'))_+}{\D_X(y,y')}.
\]
\end{proposition}
\begin{proof}
Write $F(y)$ and $G(y)$ for the second and third expressions. Then $F \le G$ holds because $\D_Y \ge \D_X$, and $G\le \vert \partial \e \vert$ on $Y$ as an immediate consequence of the definitions. We must only prove $F\ge \vert \partial \e \vert$ on $Y$. Fix $y \in Y$ with $F(y) < \infty$. Let $(\gamma_t)_{t > 0}$ be an $\mathrm{EVI}_\lambda$-flow of $\e$ in $(X,\D_X)$, starting from $y \in Y$.
We then apply the energy identity for $\mathrm{EVI}_\lambda$-flows at time $t=0$, as stated in   \cite[Theorem 3.5, equation (3.17)]{muratori2020gradient}:
\begin{equation}
\vert \partial \e \vert^2(y) = - \lim_{h\downarrow 0}\frac{\e(\gamma_{h})-\e(y)}{h}. \label{pf:slopes}
\end{equation}
Because $(\gamma_t)_{t > 0}$ is contained in $Y$, it is also an $\mathrm{EVI}_\lambda$-flow of $\e$ in $(Y,\D_X|_{Y \times Y})$ starting from $y$.
Apply the energy identity again to deduce that the right-hand side of \eqref{pf:slopes} equals $F(y)^2$.
\end{proof}

\section{Existence of geodesics} \label{se:geodesics}

In this section we give sufficient conditions for a subset $Y \subset X$ with its induced length distance to be a geodesic space. We first give a straightforward general condition in the metric space setting (Proposition \ref{pro:geodesicY}), which is powerful but  not quite strong enough to cover some of our main examples such as the Carlen-Gangbo spheres. We then give a more novel condition (Theorem \ref{thm:dualExisF}), tailored to sets in Wasserstein space defined by linear constraints.
Both results require, as an assumption, that two points in the submanifold $Y$ can be connected by some finite-length curve contained in $Y$, and in Section \ref{se:finitecurves} we give two sufficient conditions for this in the setting of the EVI invariance principle.

\subsection{Compactness criteria} \label{ssec:GeoMetric}

In this section we give a compactness-based  sufficient condition for $Y$ to be a geodesic space.

\begin{proposition} \label{pro:geodesicY}
Assume that $(X,\D_X)$ is a complete metric space, and assume there exists a Hausdorff topology $\sigma$ on $X$ satisfying:
\begin{enumerate}[(1)]
\item $\D_X$-bounded sets are sequentially $\sigma$-compact. 
\item $\D_X$ is $\sigma \times \sigma$-sequentially lower semicontinuous.
\item $Y$ is a $\sigma$-closed subset of $X$.
\end{enumerate}
If $\D_Y$ is everywhere finite on $Y \times Y$, then $(Y,\D_Y)$ is a complete geodesic space. 
\end{proposition}
\begin{proof}
Let $x,y\in Y$. 
By assumption, the formula defining $\D_Y(x,y)$ in \eqref{eq:DefDY} is finite. We will show that it is attained.
Let $( \gamma^k )_{k \geq 1}$ be a minimizing sequence such that
\[
\int_0^1|\dot\gamma^k|_t\,\d t \le \frac{1}{k} + \D_Y(x,y) =: L_k, \qquad \forall k \in \N.
\]
By reparametrization, we may assume that $\gamma^k$ is $L_k$-Lipschitz with respect to $\D_X$.
The assumptions on the topology $\sigma$ let us apply the refined Arzelà-Ascoli theorem \cite[Proposition 3.3.1]{ambrosio2008gradient}, to may find a continuous curve $\gamma: [0,1] \to (X,\D_X)$ such that 
\[ \gamma^k_t \xrightarrow[k \rightarrow +\infty]{\sigma} \gamma_t, \qquad \forall t \in [0,1], \]
up to extracting a subsequence. This curve $\gamma$ must lie in $Y$ because $Y$ is $\sigma$-closed. By the $\sigma\times\sigma$-sequential lower semicontinuity of $\D_X$, $\gamma$ must be Lipschitz with respect to $\D_X$, with constant $L:=\D_Y(x,y)$. Hence, $|\dot\gamma|_t \le L$ a.e., and we get
\[
\int_0^1|\dot\gamma|_t\,\d t \le L=\D_Y(x,y).
\]
This must be equality, and we see that $\gamma$ is a minimizer for \eqref{eq:DefDY}. Since $|\dot\gamma|_t \le L$ a.e., we must have  $|\dot\gamma|_t=L$ a.e. Hence, 
\[
\D_Y(\gamma_s,\gamma_t) \le \int_s^t|\dot\gamma|_u\,\d u \le L(t-s),
\]
and we deduce that $\gamma$ is a $Y$-geodesic from $x$ to $y$.

Repeating the above argument but with variable endpoints $(x,y)$ shows that $\D_Y$ is $\sigma_Y\times\sigma_Y$-sequentially lower semicontinuous on $\D_Y \times \D_Y$-bounded sets, where $\sigma_Y$ is the subspace topology.
As $(X,\D_X)$ is complete and $\D_X \vert_{Y \times Y} \leq \D_Y$, it follows readily that $(Y,\D_Y)$ is complete.
\end{proof}

The main example for Proposition \ref{pro:geodesicY} is when $(X,\D_X) = ( \ps_2 ( \R^d) , \W_2)$ and $Y=\F^\perp$ is defined as in \eqref{intro:Gperp} for some given set $\F$ of continuous functions with polynomial growth of order at most $p \in [1,2)$.
As suitable topology $\sigma$ is then the one generated by the Wasserstein distance $\W_p$. 
Since $p=2$ is excluded, this does not cover the Carlen-Gangbo sphere. This motivates the developments of the next section.

\subsection{Linear constraints on Wasserstein space} \label{ssec:GeoW2}

Let us briefly recall the standard definition of $p$-Wasserstein distance.
Consider a separable metric space $(X,\D_X)$, $x_0 \in X$, and $p \in [1,\infty)$. Let $\ps(X)$ denote the set of Borel probability measures on $X$, 
equipped with the usual topology of weak convergence (dual to bounded continuous functions). Let $\ps_p(X)$ denote the set of  $\mu \in \ps(X)$ with $\D_X(\cdot,x_0) \in L^p(\mu)$. The $p$-Wasserstein distance $\W_p$ is defined on $\ps_p(X)$ by
\[
\W_p(\mu,\nu) := \inf_\pi \bigg(\int_{X \times X}\D_X^p(x,y)\,\pi(\d x,\d y)\bigg)^{1/p},
\]
where the infimum is over couplings, i.e., probability measures $\pi$ on $X \times X$ satisfying $\pi(\cdot \times X) = \mu$ and $\pi(X \times \cdot) = \nu$.

Let $\phi : \R^d \rightarrow \R_+$ be a continuous function. 
Let $C(\R^d)$ be the set of continuous real-valued functions on $\R^d$. 
Define  the normed vector space
\[ C_\phi ( \R^d ) := \bigg\{ f \in C ( \R^d ) : \lVert f \rVert_\phi := \sup_{x \in \R^d} \frac{\vert f(x) \vert}{1 + \phi (x)} <  \infty \bigg\}. \]
It is immediate to check that $(C_\phi ( \R^d ), \lVert \cdot \rVert_\phi)$ is a Banach space.
We further introduce
\[ \ps_\phi ( \R^d) := \bigg\{ \rho \in \ps ( \R^d ) :  \int_{\R^d} \phi \, \d \rho <  \infty \bigg\}. \]
For a set $\F \subset C_\phi ( \R^d )$, we define
\begin{equation}
\F^\perp := \bigg\{ \rho \in \ps_\phi(\R^d) : \int_{\R^d} f\,\d \rho =0 , \ \forall f \in \F\bigg\}, \label{def:Fperp-pfsection}
\end{equation}
as well as
\begin{equation} \label{eq:W2F}
\W^2_\F (\rho_{\mathrm{i}},\rho_{\mathrm{f}}) := \inf \int_{0}^1 \int_{\R^d} \vert v_t \vert^2 \,\d \rho_t \d t,  \qquad \rho_{\mathrm{i}}, \rho_{\mathrm{f}} \in \ps_\phi (\R^d).
\end{equation}
where we minimize over $(\rho_t,v_t)_{0 \leq t \leq 1} \in C ( [0,1]; \ps ( \R^d )) \times L^2 ( [0,1] \times \R^d, \d t \rho_t(\d x))$ such that $(\rho_0,\rho_1)=(\rho_{\mathrm{i}},\rho_{\mathrm{f}})$, $\rho_t \in \F^\perp$ for all $t \in [0,1]$, and $\partial_t \rho_t + \nabla \cdot (\rho_t v_t ) = 0$ in  the sense of distributions.

\begin{theorem} \label{th:Fgeodesic}
Let $(X,\D_X)=(\ps_2(\R^d),\W_2)$. Let $\F \subset C_\phi ( \R^d )$ be such that  $\phi-c \in \F$ for some constant $c \in \R$, and such that $Y:=\F^\perp$ is nonempty.  Define $\D_Y$ as in \eqref{eq:DefDY}. Then $\W_\F \equiv \D_Y$. Moreover, if $\D_Y$ is finite on $Y\times Y$, then $Y$ is a geodesic space.
\end{theorem}

The main point of Theorem \ref{th:Fgeodesic} is the ``if-then'' statement. The problem of proving that $(\F^\perp,\W_F)=(Y,\D_Y)$ is a geodesic space is thus reduced to the much gentler problem of proving that it is a length space, i.e., that curves of finite length exist between any two points. This remaining task will be taken up in Section \ref{se:finitecurves} below.

Theorem \ref{th:Fgeodesic} will follow quickly from a more general one which results in a somewhat more readable proof, and for which we also provide a duality theorem.
With a slight abuse of notation, we identify $\phi$ with the function $(t,x) \mapsto \phi (x)$ to define the analogous Banach space $C_\phi ( [0,1] \times \R^d )$.
We will add superscripts to indicate smoothness in the usual way; mainly, $C_\phi^1 ( [0,1] \times \R^d )$ denotes the set of continuously differentiable functions in $C_\phi ( [0,1] \times \R^d )$.
For $\G \subset C_\phi ( [0,1] \times \R^d )$, we will write $\G^\perp$ for the set of Borel maps $[0,1] \ni t \mapsto \rho_t \in \ps(\R^d)$ such that $\G \subset L^1(\d t\rho_t(\d x))$ and
\[
\int_0^1\int_{\R^d}g(t,x)\,\rho_t(\d x)\d t = 0, \qquad \forall g \in \G.
\]
For any $\rho_{\mathrm{i}}, \rho_{\mathrm{f}} \in \ps_\phi (\R^d)$, we set 
\begin{equation} \label{eq:AbsInducedD}
\W^2_\G (\rho_{\mathrm{i}},\rho_{\mathrm{f}}) := \inf \int_{0}^1 \int_{\R^d} \vert v_t \vert^2 \,\d \rho_t \d t,  
\end{equation}
where we minimize over $(\rho_t,v_t)_{0 \leq t \leq 1} \in C ( [0,1]; \ps ( \R^d )) \times L^2 ( [0,1] \times \R^d, \d t  \rho_t(\d x))$ such that $(\rho_0,\rho_1)=(\rho_{\mathrm{i}},\rho_{\mathrm{f}})$, $(\rho_t)_{0 \leq t \leq 1} \in \G^\perp$, and
\begin{equation} \label{eq:minthmExis}
\partial_t \rho_t + \nabla \cdot (\rho_t v_t ) = 0, 
\end{equation} 
in the sense of distributions.

\begin{theorem}  \label{thm:dualExisF}
Let $\G \subset C_\phi ( [0,1] \times \R^d )$ be a linear space such that $\G^\perp$ is non-empty. Assume there exists $(\varphi^0,g^0) \in C^1_\phi ([0,1]\times\R^d;\R) \times \G$ such that $\nabla \varphi^0 \in C_{\sqrt{\phi}}([0,1]\times\R^d;\R^d)$ and 
\begin{equation} \label{eq:Qualif}
\sup_{(t,x) \in [0,1] \times \R^d} \frac{2 \big( \partial_t \varphi^0 (t,x) + g^0 (t,x) \big) + \vert \nabla \varphi^0 (t,x) \vert^2}{1+\phi(x)} < 0.
\end{equation}
Then, for any $\rho_{\mathrm{i}}, \rho_{\mathrm{f}} \in \ps_\phi (\R^d)$, 
\[ \frac{1}{2} \W^2_\G (\rho_{\mathrm{i}},\rho_{\mathrm{f}}) = \sup_{\substack{g \in \G, \, \varphi \in  C_\phi^1 ([0,1] \times \R^d) \\ \partial_t \varphi + \tfrac{1}{2} \vert \nabla \varphi \vert^2 + g \leq 0}} \int_{\R^d} \varphi (1,\cdot)\, \d \rho_{\mathrm{f}}  - \int_{\R^d} \varphi (0,\cdot) \,\d \rho_{\mathrm{i}}.\]
Moreover, a  minimizer in \eqref{eq:AbsInducedD} exists if $\W^2_\G (\rho_{\mathrm{i}},\rho_{\mathrm{f}})< \infty$.
\end{theorem}

\begin{proof}[Proof of Theorem \ref{th:Fgeodesic} using Theorem \ref{thm:dualExisF}]
Set $\G = \{ a f : a \in C ([0,1]), \, f \in \mathrm{span}(\F) \}$, with $\mathrm{span}(\F)$ being the closed linear span in $(C_\phi(\R^d),\|\cdot\|_\phi)$. It is clear that $\G^\perp$ is the set of Borel maps $[0,1] \to \F^\perp$, and so $\W_\G=\W_\F$. 
The fact that $\W_\F=\D_Y$ follows by combining \eqref{eq:DefDY-squared} with the characterization of AC curves in Wasserstein space \cite[Theorem 8.3.1]{ambrosio2008gradient}, as distributional solutions of the continuity equation $\partial_t \rho_t + \nabla \cdot ( \rho_t v_t ) = 0$, with the identification of metric derivative $|\dot\rho|_t$ as the minimal $\|v_t\|_{L^2(\rho_t)}$ in a suitable sense. With this identification it is also straightforward to argue that any minimizer $(\rho_t)_{0 \le t \le 1}$ in the definition of $\W_\F$ can be reparametrized to have constant metric derivative, making it a $Y$-geodesic.
Theorem \ref{thm:dualExisF} applies after noting that the assumption \eqref{eq:Qualif} holds with $g^0(t,x)=c-1-\phi(x)$ and $\varphi^0(t,x)=-ct$.
This implies that $Y$ is a geodesic space if $\D_Y < \infty$. 
\end{proof}

To prove Theorem \ref{thm:dualExisF}, we closely follow the approach of \cite[Chapter 4]{brenier2020examples}. We first relax the geodesic problem by looking at pairs of finite Radon measures $(\rho,q) \in (C_\phi ( [0,1] \times \R^d) )^\star \times (C_{\sqrt{\phi}} ( [0,1] \times \R^d; \R^d ) )^\star$.
We consider the problem of minimizing the action functional
\[ \A ( \rho, q ) := \sup_{\substack{(A,B) \in C_\phi ( [0,1] \times \R^d) \times C_{\sqrt{\phi}} ([0,1] \times \R^d; \R^d) \\ 2 A + \vert B \vert^2 \leq 0 }} \int_{[0,1] \times \R^\d} A \, \d \rho + \int_{[0,1] \times \R^\d} B  \cdot \d q , \]
under the constraints:
\begin{enumerate}[(1)]
    \item In weak sense against bounded $C^1$ test functions,
    \begin{equation} \label{eq:RelaxTrans}
     \begin{cases}
    \partial_t \rho + \nabla \cdot q = 0, \\
    (\rho_0, \rho_1) = (\rho_i, \rho_f).
    \end{cases}
    \end{equation}
    \item For every $
    g \in \G$, 
    \begin{equation} \label{eq:constraintsF}
    \int_{[0,1] \times \R^d} g \, \d \rho = 0. 
    \end{equation}
\end{enumerate}
This problem is convex and lower semicontinuous with respect to the weak-$\star$ topology.  
We borrow the following computation of  $\A ( \rho, q)$ from \cite[Proposition 3.4]{brenier1997homogenized}.

\begin{lemma}
We have $\A ( \rho, q ) < +\infty$ if and only if $\rho \ge 0$ and $q$ has a square-integrable density with respect to $\rho$. In this case,
\[ \A (\rho,q) = \frac{1}{2} \int_{[0,1] \times \R^d} \bigg\vert \frac{\d q}{\d \rho} \bigg\vert^2 \d \rho < +
\infty. \]
\end{lemma}

Considering test functions in \eqref{eq:RelaxTrans} that only depend on time, we get that the time-marginal of $\rho$ is the Lebesgue measure.
For any disintegration $\d \rho = \d \rho_t \d t$, we similarly get that $\rho_t$ has mass $1$ for a.e.\ $t \in [0,1]$.
Hence, if \eqref{eq:RelaxTrans} and \eqref{eq:constraintsF} both hold, and $\A ( \rho, q ) < \infty$, then $(\rho_t)_{0 \leq t \leq 1} \in \G^\perp$.
Conversely, \eqref{eq:minthmExis} implies that $q(\d t,\d x) := v_t(x) \rho_t(\d x) \d t$ belongs to $( C_{\sqrt{\phi}} ( [0,1] \times \R^d; \R^d ) )^\star$ if $\A (\rho,q) < \infty$.
We have thus shown that
\begin{equation}
\frac{1}{2} \W^2_{\G} ( \rho_\mathrm{i}, \rho_{\mathrm{f}} ) = \inf\big\{  \A(\rho,q) : (\rho,q) \text{ satisfies \eqref{eq:RelaxTrans}--\eqref{eq:constraintsF}} \big\} . \label{pf:eq:WG-A}
\end{equation}
We now turn to the proof of Theorem \ref{thm:dualExisF}.

\begin{proof}[Proof of Theorem \ref{thm:dualExisF}]
Let us introduce the Banach space $E := C_\phi ( [0,1] \times \R^d) \times C_{\sqrt{\phi}} ( [0,1] \times \R^d; \R^d)$.
For $(A,B) \in E$, we set
\[ K_1 (A,B) := - \int_{\R^d} \varphi (1,x) \,\rho_{\mathrm{f}} (\d x) + \int_{\R^d} \varphi (0,x)\,\rho_{\mathrm{i}} (\d  x), \]
if there exist $g \in \G$ and $\varphi \in C_\phi^1 ([0,1] \times \R^d)$ such that
\[ A = \partial_t \varphi + g, \qquad B = \nabla \varphi, \]
and otherwise we set $K_1 (A,B) := \infty$.
Let us check the consistency of this definition: If $(A,B)$ is given by two different couples $(\varphi,g)$ and $(\tilde\varphi,\tilde{g})$, then $\nabla \varphi = \nabla \tilde\varphi$, which implies that $h' := g-\tilde g = \partial_t\tilde\varphi-\partial_t\varphi$ is independent of $x$. Since $g-\tilde{g} \in \G$, we may choose arbitrarily $(\rho_t)_{0 \leq t \leq 1}$ in $\G^\perp$ (which was assumed nonempty), and then
\begin{align*}
0 &= \int_0^1 \int_{\R^d} \big( g - \tilde{g} \big) \, \d\rho_t  \d t = \int_0^1 h'(t) \,\d t = h (1) - h(0) \\
    &= \int_{\R^d} \big(\tilde\varphi(1,x)-\varphi(1,x)\big) \rho_{\mathrm{f}}(\d x) - \int_{\R^d} \big(\tilde\varphi(0,x)-\varphi(0,x)\big) \rho_{\mathrm{i}}(\d x).
\end{align*}
Thus, $K_1(A,B)$ does not depend on the chosen decomposition for $(A,B)$.
We then define
\[   K_2 (A,B) :=
\begin{cases}
    0 \quad &\text{if} \quad 2 A + \vert B \vert^2 \leq 0 \\
     \infty &\text{otherwise}.
\end{cases}
\]
for $(A,B) \in E$.
The functions $K_1$ and $K_2$ are convex, and we now apply the Fenchel-Rockafellar duality theorem \cite[Theorem 1.12]{brezis2011functional}.
To do so, it is sufficient to exhibit $u^0 \in E$ such that both $K_1$, $K_2$ are finite at $u^0$, and $K_2$ is continuous at $u^0$.
Using \eqref{eq:Qualif}, the choice $(A^0,B^0) = (\partial_t \varphi^0 + g^0,\nabla \varphi^0)$ is suitable, since any small perturbation $(A,B)$ of $(A^0,B^0)$ in $E$-norm still satisfies $2 A + \vert B \vert^2 < 0$.
Fenchel-Rockafellar then states that 
\[ \sup_{u \in E} - K_1 (u) - K_2 (u) = \inf_{v \in E^\star} K_1^\star (-v) + K_2^\star (v), \]
with the infimum on the right-hand side being attained as long as it it finite. 
We then notice that the Legendre transform $K_2^\star$ is precisely equal to $\A$, and that
\begin{equation*}
K^\star_1 (-\rho,-q) = \sup_{\varphi,g} \int_{[0,1] \times \R^d} \Big(( - \partial_t \varphi - g )\d \rho -  \nabla \varphi \cdot \d q\Big) + \int_{\R^d} \varphi (1,\cdot) \d \rho_{\mathrm{f}}  - \int_{\R^d} \varphi (0,\cdot) \d \rho_{\mathrm{i}} .
\end{equation*} 
For $(\rho,q)$ such that $\A(\rho,q) < \infty$ we see that 
\[
K^\star_1 (-\rho,-q) = \begin{cases} 0 &\text{if }  \partial_t \rho + \nabla \cdot q = 0, \quad (\rho_0,\rho_1) = ( \rho_\mathrm{i}, \rho_\mathrm{f} ), \quad (\rho_t)_{0 \leq t \leq 1} \in \G^\perp \\ \infty &\text{otherwise}. \end{cases}
\]
Recalling \eqref{pf:eq:WG-A}, this proves the desired strong duality.
\end{proof}

\subsection{Existence of curves of finite length} \label{se:finitecurves}

Recall that Proposition \ref{pro:geodesicY} and Theorem \ref{th:Fgeodesic} proved that certain spaces are geodesic provided that there exist curves of finite length. In this section we show how to construct such curves in the setting of the EVI invariance principle for $\lambda > 0$, where we have a natural candidate: For $y_1,y_2 \in Y$, we can run the EVI flow starting from each point, which should converge to the unique minimizer of the functional. We then concatenate the EVI flow started from $y_1$ with the time-reversal of the EVI flow started from $y_2$, to obtain a curve connecting $y_1$ and $y_2$. It appears to be an open question whether or not an $\mathrm{EVI}_\lambda$-flow for $\lambda > 0$ necessarily produces a curve of finite length. We show in Proposition \ref{pr:Ylengthspace} that this is true for starting points in the domain of the energy functional $\e$. To handle starting points outside the domain of the functional, we require stronger information on the short-time regularization of the flow, which does not appear to follow solely from the EVI structure, and this is treated in  Proposition \ref{pr:Ylengthspace2}.

\begin{proposition} \label{pr:Ylengthspace}
Assume the conditions of Theorem \ref{th:main} hold, with $\lambda > 0$, and with $(X,\D_X)$ being a complete (non-extended) metric space.  
Assume there exists $y_* \in Y$ such that $\e(y_*) =\inf_{x \in X}\e(X)$.
Then every pair $y_1,y_2 \in Y \cap D(\e)$ can be connected by a curve of finite length in $Y$; that is, $\D_Y(y_1,y_2) < \infty$. 
\end{proposition}

\begin{rem}
Note that the assumption involving $y_*$ cannot be removed in general, as illustrated by the simple example where $X=\R$, $\e(x)=x^2$, and $Y = \R \setminus \{0\}$.
However, if $Y \subset X$ is assumed to be closed and $\e$ lower semicontinuous, then the assumption involving $y_*$ can be removed. 
Indeed, by \cite[Theorem 3.5]{muratori2020gradient} and completeness of $(X,\D_X)$, the functional $\e$ automatically admits a unique minimizer on $X$, and every $\mathrm{EVI}_\lambda$ flow converges to it. In particular, an $\mathrm{EVI}_\lambda$-flow initialized in $Y$ stays in $Y$ by the assumption of Theorem \ref{th:main}, and by closedness of $Y$ we deduce that $y_* \in Y$.
\end{rem}

\begin{proof}[Proof of Proposition \ref{pr:Ylengthspace}]
Let $(\gamma^i_t)_{t \ge 0}$ be the $\mathrm{EVI}_\lambda$-flow of $\e$ started from $y_i$, for each $i=1,2$. Since $t \mapsto \e(\gamma^i_t)$ is nonincreasing and converges to $\e(y_*)$ \cite[Theorem 3.5]{muratori2020gradient}, the energy identity \eqref{eq:energyidentity} yields
\[
\int_0^\infty |\dot\gamma^i|^2_t\,\d t = \e(y_i) - \e(y_*) < \infty.
\]
We also have 
\[
|\dot\gamma^i|_t=|\partial\e|(\gamma^i_t) \le |\partial\e|(\gamma^i_{t_0})e^{-\lambda(t-t_0)} < \infty, \quad t > t_0 > 0,
\]
where the first equality follows again from (3.17) of \cite{muratori2020gradient}, and the two inequalities follow from  the regularization estimate (3.18d) therein. By using this to control the integral over $[t_0,\infty)$, and noting that $L^1[0,t_0] \subset L^2[0,t_0]$, we deduce that
\[
\int_0^\infty |\dot\gamma^i|_t\,\d t < \infty, \quad i=1,2.
\]
Now, let $h : [0,1/2) \to [0,\infty)$ be any strictly increasing continuously differentiable bijection. Define a curve $\gamma : [0,1] \to X$ by setting
\[
\gamma_t = \begin{cases}
    \gamma^1_{h(t)} &\text{if } 0 \le t < 1/2 \\
    y_* &\text{if } t=1/2 \\
    \gamma^2_{h(1-t)} &\text{if } 1/2 < t \le 1.
\end{cases}
\]
Then $\gamma \in AC([0,1];(Y,\D_X|_{Y\times Y}))$ satisfies $\gamma_0=y_1$ and $\gamma_1=y_2$. Moreover,
\begin{align*}
\int_0^1 |\dot\gamma|_t\,\d t 
    &= \int_0^\infty |\dot\gamma^1|_t\, \d t + \int_0^\infty |\dot\gamma^2|_t\, \d t  < \infty. \qedhere
\end{align*}
\end{proof}

If we do not assume $y_i \in D(\e)$, the proof of Proposition \ref{pr:Ylengthspace} fails only because we do not have a way to guarantee that $|\dot\gamma^i|_t$ is integrable near the origin. The energy identity \eqref{eq:energyidentity} suggests another way to proceed. The metric slope $|\partial\e|(\gamma^i_t)$ is finite for all $t > 0$, even if it is not so at $t=0$. Using solely the EVI structure, one can obtain a regularization estimate of $|\partial\e|(\gamma^i_t) \lesssim 1/t$ for short time; see, e.g., \cite[(3.18d)]{muratori2020gradient}.
However, this is not good enough for our purposes, as it is not integrable for $t \downarrow 0$.
Interestingly, however, in typical Wasserstein space settings, we can show  a sharper estimate of $|\partial\e|(\gamma^i_t) \lesssim 1/\sqrt{t}$; this is noted in the Ornstein-Uhlenbeck setting in \cite[Remark 24.21]{villani2008optimal}. 
This comes from a short-time regularization bound for the Fisher information along Fokker-Planck equations,  which we justify and exploit in the following two propositions. The first  treats Euclidean space via shift-Harnack inequalities, and the second treats Riemannian  manifolds  of nonnegative Ricci curvature by means of classical  Li-Yau inequalities and a curvature-dimension condition. 
This points to an interesting open problem, which is to characterize when this stronger regularization rate holds for EVI flows in the metric space setting.

\begin{proposition} \label{pr:Ylengthspace2}
Let $(X,\D_X)=(\ps_2(\R^d),\W_2)$ and $\lambda > 0$, and consider $\e : X \to (-\infty,\infty]$ of the form
\[
\e(\rho) = H(\rho) + \int V\,\d\rho,
\]
where $V : \R^d \to \R$ is a $\lambda$-convex function with bounded second derivatives.
Assume the conditions of Theorem \ref{th:main} hold, with this choice of $X$, $\lambda$, and $\e$, and for some $Y \subset X$. Assume also that there exists $y_* \in Y$ such that $\e(y_*) =\inf_{x \in X}\e(X)$. Then every pair $y_1,y_2 \in Y$ can be connected by a curve of finite length in $Y$; that is, $\D_Y(y_1,y_2) < \infty$.  
\end{proposition}
\begin{proof}
Let us assume without loss of generality that $\pi(\d x)=e^{-V(x)}\d x$ is a probability measure.
The $\mathrm{EVI}_\lambda$-flow of $\e=H(\cdot\,|\,\pi)$ in $X=\ps_2(\R^d)$ is well known \cite[Theorem 11.2.8]{ambrosio2008gradient} to be governed by the Fokker-Planck equation,
\begin{equation}
\partial_t \rho = \mathrm{div}(\rho\nabla V) + \Delta \rho. \label{pf:FKP-Ylength}
\end{equation}
Consider a distributional solution $(\rho_t)_{t \ge 0} \in C([0,\infty);\ps_2(\R^d))$, where the initial $\rho_0$ can be any element of $\ps_2(\R^d)$, not necessarily absolutely continuous.
We will show  that
\begin{equation}
    \int_0^1 |\dot\rho|_t\,\d t < \infty. \label{pf:rho-finitelength}
\end{equation}
Once this is shown, we follow the proof of Proposition \ref{pr:Ylengthspace} to deduce that $(\rho_t)_{t \ge 0}$ necessarily has finite length and to construct a path of finite length between two arbitrary points in $Y$.

For each $t > 0$ we know that $\rho_t$ has a strictly positive density, by standard parabolic regularity theory.
The key fact we will use is that there is a constant $c$ such that
\begin{equation}
I(\rho_t) := \int |\nabla \log \rho_t|^2\,\d\rho_t \le c/t, \quad \text{for all } 0 < t < 1, \label{pf:FIregularize}
\end{equation}
where $c$ depends only on $\lambda$, the uniform bounds of $\nabla^2 V$, the dimension $d$, and $\int |x|^2\rho_0(\d x)$. Crucially, $c$ is finite regardless of the regularity of $\rho_0$.
Let us take \eqref{pf:FIregularize} for granted for the moment. We then use
\[
|\dot\rho|_t^2 = |\partial \e|^2(\rho_t) = \int \bigg|\nabla \log \frac{\rho_t}{\pi}\bigg|^2\,\d\rho_t,
\]
where the first formula is the EVI energy identity \eqref{eq:energyidentity}, and the second is standard \cite[Theorem 10.4.9]{ambrosio2008gradient}. This implies
\[
|\dot\rho|_t \le \sqrt{I(\rho_t)} + \Big(\int \big|\nabla V\big|^2\,\d\rho_t\Big)^{1/2}.
\]
Note that $\nabla V$ has linear growth, as $\nabla^2 V$ is bounded.
Hence, the second term is bounded by a constant times $1+\int|x|^2\,\rho_0(\d x)$, by the triangle inequality and the contraction estimate $\W_2(\rho_t,\pi) \le e^{-\lambda t}\W_2(\rho_0,\pi)$ \cite[Theorem 3.5]{muratori2020gradient}. Using \eqref{pf:FIregularize}, we thus obtain
\[
|\dot\rho|_t \le c(1+t^{-1/2}), \quad \text{for all } t > 0,
\]
for some $c$. The claim \eqref{pf:rho-finitelength} follows.

It remains to justify \eqref{pf:FIregularize}. One way is using F.-Y. Wang's shift Harnack inequalities, developed in great generality in  \cite{wang2014integration}. A concise recent reference is \cite{altschuler2023shifted}, and Theorem 1.1 therein contains exactly our setting. Indeed, let $p_t(x,y)$ denote the transition density of $X_t$ given $X_0=x$ associated with the Langevin SDE
\[
dX_t = -\nabla V(X_t) \d t + \sqrt{2}\d B_t.
\]
In PDE terms, $p_t(x,\cdot)$ is the (fundamental) solution of \eqref{pf:FKP-Ylength} initialized from $\delta_x$.
Let $\beta > 0$ be such that $\nabla^2 V(x) \le \beta I$ for all $x$.
Using the  implication (LGE) $\Rightarrow$ (SRT$_q$) for $q=2$ in \cite[Theorem 1.1]{altschuler2023shifted}, for all $x,v \in \R^d$ we have
\[
\log \int_{\R^d} \Big(\frac{p_t(x,z-v)}{p_t(x,z)}\Big)^2p_t(x,z)\,\d z \le \frac{\beta |v|^2}{1-\exp(-2\beta t)} .
\]
Rearranging, we have
\[
\int_{\R^d} \Big(\frac{p_t(x,z-v)-p_t(x,z)}{p_t(x,z)}\Big)^2p_t(x,z)\,\d z \le \exp\bigg(\frac{\beta |v|^2}{1-\exp(-2\beta t)}\bigg)-1.
\]
By Fatou's lemma,
\begin{align}
\int_{\R^d} \big(v\cdot \nabla \log_z p_t(x,z)\big)^2 p_t(x,z)\,\d z &\le \liminf_{h \downarrow 0} \frac{1}{h^2}\int_{\R^d} \Big(\frac{p_t(x,z-hv)-p_t(x,z)}{  p_t(x,z)}\Big)^2p_t(x,z)\,\d z \\
    &\le \frac{\beta |v|^2}{1-\exp(-2\beta t)}.
\end{align}
Sum over $v$ in an orthonormal basis to get
\[
I(p_t(x,\cdot)) \le \frac{d\beta}{1-\exp(-2\beta t)}.
\]
By the well-known convexity of Fisher information,
\[
I(\rho_t) = I\bigg(\int_{\R^d} p_t(x,\cdot)\,\rho_0(\d x)\bigg) \le \frac{d\beta}{1-\exp(-2\beta t)}.
\]
This proves \eqref{pf:FIregularize}. Another way to prove \eqref{pf:FIregularize} would be to use a Bismut-type formula for $\nabla \log \rho_t$, which we learned from \cite[Remark 4.13]{follmer2006time}:
\[
\nabla\log\rho_t(x) = -\frac{1}{\sqrt{2} \,t}\E\bigg[\int_0^t \big(I - s\nabla^2 V(X_s)\big)\,\d B_s \,\Big|\, X_t=x \bigg].
\]
Taking the square, integrating with respect to $\rho_t(\d x)$, and applying Jensen's inequality and It\^o isometry leads again to \eqref{pf:FIregularize}.
\end{proof}

We next treat the setting of a  weighted Riemannian manifold $(M,g,\pi)$, in the sense of \cite[Definition 3.17]{grigoryan2009heat}. That is, we assume that $(M,g)$ is a smooth orientable Riemannian manifold with finite dimension $d \in \N$ and volume measure $\mathrm{vol}$, and $\pi = e^{-V} \mathrm{vol}$ is a positive measure on $M$ for a smooth potential $V : M \rightarrow \R$.
We further assume that $M$ is complete, and $\pi$ has finite mass. For $n \in(d,\infty]$ and $\lambda\in\R$, we say that $(M,g,\pi)$ satisfies the $CD(\lambda,n)$ condition (curvature-dimension) if
\begin{equation} \label{eq:CDCondition}
\mathrm{Ric} + \nabla^2 V \geq \lambda I + \frac{\nabla V \otimes \nabla V}{n - d},
\end{equation}
in the sense of quadratic forms, where $\mathrm{Ric}$ is the Ricci curvature tensor. The differential entropy $H(\rho)$ is defined as usual by $H(\rho) = \int \rho\log\rho \,\d \mathrm{vol}$, when $\rho\log \rho \in L^1(\mathrm{vol})$, and $H(\rho)=\infty$ otherwise.

\begin{proposition} \label{pr:RiemannCD}
Suppose $(M,g,\pi)$ is a weighted Riemannian manifold as described above, satisfying the condition $CD(\lambda,n)$ for some finite $n > d$ and $\lambda > 0$.
Let $(X,\D_X)=(\ps_2(M),\W_2)$, and consider $\e : X \to (-\infty,\infty]$ of the form
\[
\e(\rho) = H(\rho) + \int V\,\d\rho.
\]
Assume the conditions of Theorem \ref{th:main} hold, with this choice of $X$, $\lambda$, and $\e$, and for some $Y \subset X$. Assume also that there exists $y_* \in Y$ such that $\e(y_*) =\inf_{x \in X}\e(X)$. Then every pair $y_1,y_2 \in Y$ can be connected by a curve of finite length in $Y$; that is, $\D_Y(y_1,y_2) < \infty$.  
\end{proposition}
\begin{proof}
Assume without loss of generality that $\pi(\d x)=e^{-V(x)}\d x$ is a probability measure, so that $\e=H(\cdot\,|\,\pi)$.
We notice that $CD(\lambda,n)$ implies $CD(\lambda,\infty)$, which is equivalent to $H(\cdot \vert \pi)$ being geodesically $\lambda$-convex in $\ps_2 (M)$ by \cite[Corollary 1.5]{sturm2005convex}.
It is shown in \cite[Theorem 23.19]{villani2008optimal} that the $\mathrm{EVI}_\lambda$-flow\footnote{The notion of gradient flow given by \cite[Definition 23.7]{villani2008optimal} is slightly different from ours, but it implies ours when $\e$ is $\lambda$-convex, as noted in \cite[Remark 23.8]{villani2008optimal}.} for $H(\cdot\,|\,\pi)$ is described by the Fokker-Planck equation 
\begin{equation} 
\partial_t \rho_t = \mathrm{div}(\rho_t \nabla V) + \Delta \rho_t = L^\star \rho_t, \label{pf:FKP-Riemannian}
\end{equation}
where $L := - \nabla V \cdot \nabla + \Delta$,
provided that we check first that the solution $(t,x) \mapsto \rho_t (x)$ of \eqref{pf:FKP-Riemannian} is $C^{1,2}$ and satisfies $\vert \partial \e \vert (\rho _t) < \infty$ for $t > 0$.
To check these requirements, we notice that \eqref{pf:FKP-Riemannian} is equivalent to the heat equation
\begin{equation} 
\partial_t \pi_t = \Delta_\pi \pi_t, \quad \text{for} \quad \pi_t :=  \d \rho_t/\d \pi, \label{pf:weighted-heat-eq}
\end{equation}
where $\Delta_\pi := \mathrm{div}_\pi \circ \nabla = L$ is the weighted Laplacian on $(M,g,\pi)$, recalling that $\mathrm{div}_\pi$ is the $L^2$-adjoint of $- \nabla$ in $L^2 ( \pi)$. 
The smoothness of $\pi_t$ for initial data $\pi_0 \in L^1 (M)$ is known \cite[Theorem 7.19]{grigoryan2009heat}.
Let $( P_t )_{t \geq 0}$ denote the associated heat semigroup, so that $\pi_t = P_t \pi_0$.

Since $CD(\lambda,n)$ implies $CD(\rho,n)$ for every $\rho \in [0,\lambda]$, we can apply the Li-Yau inequality in the form of \cite[Proposition 5.1(i)]{bakry2017li} with the parameters $\alpha \downarrow 0$ and $\rho \downarrow 0$, which proves that for any smooth bounded and non-negative $f : M \rightarrow \R$,
\[ - L \log P_t f = \frac{\vert \nabla P_t f \vert^2}{( P_t f )^2} - \frac{L P_t f}{P_t f} \leq \frac{n}{2 t}. \]  
As a consequence, integrating by parts,
\[
\int_M \vert \nabla \log P_t f \vert^2 P_t f \,\d \pi = \int_M \nabla \log P_t f \cdot \nabla P_t f  \,\d \pi = \int_M [- L \log P_t f] P_t f \,\d \pi \le \frac{n}{2t} \int_M f \, \d \pi.
\]
When $f=\pi_0$ we have $\rho_t(\d x)=P_tf(x)\pi(\d x)$, and the left-hand side is exactly $\vert \partial \e \vert^2(\rho_t)$; indeed, the equality of Fisher information and the upper gradient of relative entropy is well known in great generality under the CD($\lambda$,$\infty$) condition \cite[Proposition 4.3]{ambrosio2014calculus}.
Up to a regularization argument, noting that $\vert \partial \e \vert^2$ is lower semicontinuous,
this gives the bound $\vert \partial \e \vert^2 (\rho_t) \leq n/ 2t$ provided that $\rho_0 \ll \pi$.

We next remove the restriction to $\rho_0 \ll \pi$.
From \cite[Theorem 7.13]{grigoryan2009heat}, we have
\[ P_t f (x) = \int_M p_t (x,y) f(y) \, \d \pi (y), \]
for a smooth symmetric kernel $(t,x,y) \mapsto p_t (x,y)$. 
Under $CD(\lambda,n)$, \cite[Lemma 4.2 and Corollary 4.3]{bakry2017li} provides $C>0$ such that for every $f \ge 0$,
\begin{equation} \label{eq:BoundSG}
\lVert P_t f \rVert_\infty \leq C\big(1 +  t^{-n/2}\big) \int_M f \, \d \pi, \qquad t > 0, 
\end{equation} 
which in turn implies an $L^\infty$-bound on $p_t$.
This allows for extending the heat semigroup to general initial data $\rho_0 \in \ps(M)$ by setting
\[ P_t \rho_0 (x) = \int_M p_t (x,y) \rho_0 (\d y). \]
The $L^\infty$-bound on $p_t$ implies that $(t,x) \mapsto P_t \rho_0 (x)$ is a smooth solution of the weighted heat equation \eqref{pf:weighted-heat-eq}. Moreover, $P_t\rho_0(x)\pi(\d x)$ converges weakly to $\rho_0$ as $t\downarrow 0$. To see this, let $f : M \to \R$ be bounded and continuous, and note that $\|P_t f\|_\infty\le\|f\|_\infty$ with $P_tf\to f$ pointwise by \cite[Theorem 7.16]{grigoryan2009heat}. Thus, using the symmetry $p_t(x,y)=p_t(y,x)$,
\[
\int_M f (x) P_t \rho_0 (x) \, \pi(\d x) = \int_M \bigg( \int_Mf(x) p_t (y,x) \,\pi(\d x) \bigg) \rho_0 (\d y ) \to \int_M f \,\d \rho_0.
\]
Taking $\rho^k_0 \ll \pi$ with $\rho^k_0 \to \rho_0$ weakly as $k\to\infty$, we have $\rho^k_t:=P_t\rho^k_0 \to P_t\rho_0=\rho_t$, and by lower semicontinuity $\vert \partial \e \vert^2 (\rho_t) \le \liminf_k \vert \partial \e \vert^2 (\rho^k_t) \le n/2t$.

Finally, we combine the EVI energy \eqref{eq:energyidentity} with the preceding estimate to get $|\dot\rho|_t = \vert \partial \e ( \rho_t ) \vert \le \sqrt{n/2t}$. This implies  the short-time integrability
\begin{equation}
    \int_0^1 |\dot\rho|_t\,\d t < \infty.
\end{equation}
The rest of the argument proceeds as in the proof of Proposition \ref{pr:Ylengthspace}.
\end{proof}

\section{Applications of the EVI invariance principle} \label{se:examples}

Throughout this section we will apply results of Section \ref{se:metric} in the case $(X,\D_X)=(\ps_2(\R^d),\W_2)$.

\subsection{Carlen-Gangbo spheres and moment constraints} \label{se:CarlenGangbo}

The paper of Carlen-Gangbo \cite{carlen2003constrained} gave a detailed account of the induced Wasserstein geometry of the \emph{sphere-like} sets
\[
\mathcal{S}_{u,\theta} := \bigg\{\rho \in \ps_2(\R^d) : \int x\,\rho(\d x) = u, \ \int |x-u|^2\,\rho(\d x) = d\theta\bigg\}, \ \ u \in \R^d, \ \theta > 0,
\]
discussed in Section \ref{intro:se:CarlenGangbo}.
Using an explicit description of the geodesics of $\mathcal{S}_{u,\theta}$, they showed in \cite[Theorem 3.6]{carlen2003constrained} that $H$ is geodesically $(1/\theta)$-convex in $\mathcal{S}_{u,\theta}$.
We give a short new proof using Theorem \ref{th:main} by means of a simple invariance property for the Ornstein-Uhlenbeck dynamics, essentially already presented in Section \ref{intro:se:CarlenGangbo}.

\begin{theorem} \label{th:CGentropy}
The functional $H$ is geodesically $(1/\theta)$-convex in $\mathcal{S}_{u,\theta}$. Moreover, $\mathcal{S}_{u,\theta}$ is a geodesic space.
\end{theorem}
\begin{proof}
The second paragraph of Section \ref{intro:se:CarlenGangbo} already explained the proof that  $H$ is geodesically $(1/\theta)$-convex in $\mathcal{S}_{u,\theta}$, by applying the EVI invariance principle Theorem \ref{th:main} with $\e=H(\,|\,N_{u,\theta I})$.
It remains to show that $\mathcal{S}_{u,\theta}$ is a geodesic space.
We first apply Proposition \ref{pr:Ylengthspace} to deduce that curves of finite length exist between every two points of $Y=\mathcal{S}_{u,\theta}$. We then apply Theorem \ref{th:Fgeodesic}: Take $\F$ be the set containing the function $x\mapsto |x|^2-\theta d$ and the linear functions $x \mapsto v \cdot (x-u)$, for $v$ ranging over a basis of $\R^d$. Set $\phi(x)=|x|^2$ and $c=\theta d$. Then Theorem \ref{th:Fgeodesic} applies, and we deduce both that $\D_Y=\W_\F$ and that $\mathcal{S}_{u,\theta}$ is a geodesic space.  
\end{proof}

The fact that $\{\rho \in \mathcal{S}_{u,\theta} : \rho \ll \mathrm{Leb}\}$ is a geodesic space was shown in \cite[Section 3]{carlen2003constrained} by explicitly constructing geodesics. Theorem \ref{th:CGentropy} extends this to all of $\mathcal{S}_{u,\theta}$. It is worth noting that \cite[Theorem 3.6]{carlen2003constrained} also proved strict convexity results for non-logarithmic internal energy functionals, which seem to be out of reach of the EVI invariance principle.

It is straightforward to adapt the proof to show a number of variations of Theorem \ref{th:CGentropy}.
For example, $H$ is $(1/\theta)$-convex in the submanifold of measures with the same mean and covariance matrix as $N_{u,\Sigma}$, whenever $\Sigma$ has minimal eigenvalue $\theta$. Here  $N_{u,\Sigma}$ denotes the Gaussian measure with mean $u$ and covariance matrix $\Sigma$. When $\Sigma = \mathrm{Id}$, this recovers  \cite[Corollary 5.11]{burger2025covariance}, and by applying Proposition \ref{pr:HWI} we also recover their HWI inequality for the induced metric \cite[Proposition 5.14]{burger2025covariance}.
The following theorem further generalizes this example. Along the way, we apply Theorem \ref{thm:dualExisF} to show the existence of geodesics, answering a question left open in \cite[Page 12]{burger2025covariance}.

\begin{theorem} \label{th:CGpolynomials}
Let $u \in \R^d$, and let $\Sigma$ denote a symmetric positive definite matrix with smallest eigenvalue $\theta$.
Let $q \ge 2$, and let $\mathcal{F}_q$ denote the set of polynomials on $\R^d$ of degree at most $q$ which have mean zero under $N_{u,\Sigma}$. 
Accordingly, define
\[
\mathcal{F}_q^\perp:= \bigg\{\rho \in \ps_2(\R^d) : \mathcal{F} \subset L^1(\rho), \ \int f\,\d\rho = 0, \ \forall f \in \mathcal{F}_q\bigg\}.
\]
Then $H$ is geodesically $(1/\theta)$-convex in $\mathcal{F}_q^\perp$, and $\mathcal{F}_q^\perp$ is a geodesic space.
\end{theorem}
\begin{proof}
Defining
\[
\e(\rho) := H(\rho\,|\,N_{u,\Sigma}) = H(\rho) + \frac12\int \big\langle x-u,\Sigma^{-1}(x-u)\big\rangle\,\rho(\d x) + \frac{d}{2}\log(2\pi) + \frac12\log\det\Sigma.
\]
The terms after $H(\rho)$ are constant on $\rho \in \mathcal{F}_q^\perp$, so it suffices to prove the $(1/\theta)$-convexity of $\e$ on $\mathcal{F}_q^\perp$. Define  $Lf(x) = \Delta f(x)-  \langle x,\Sigma^{-1}\nabla f(x) \rangle$. The $\mathrm{EVI}_{1/\theta}$-flow $(\rho_t)_{t \ge 0}$ for $\e$ satisfies the Fokker-Planck equation $\partial_t \rho_t = L^*\rho_t$.
To show that this EVI flow leaves $\F_q^\perp$ invariant, the key point is that $Lf \in \F_q$ for any $f \in \F_q$. Indeed, $Lf$ is clearly a polynomial of degree at most $q$, and $\int Lf\,\d N_{u,\Sigma}=0$ because $N_{u,\Sigma}$ is stationary for the Fokker-Planck equation.
The set $\F_q$ is spanned by finitely many functions, say $f_1,\ldots,f_m$, and it follows that the vector $g(t)=(\langle \rho_t,f_i\rangle)_{i=1,\ldots,m}$ satisfies a linear system of ordinary differential equations (ODEs). If $\rho_0=N_{u,\Sigma}$, then $g(t)=0$ for all $t \ge 0$ by stationarity of $N_{u,\Sigma}$. Hence, if $\rho_0$ satisfies $\int f_i\,\d\rho=0$ for all $i$, then $g(t)$ satisfies the same ODE system with the same initialization $g(0)=0$, and thus again $g(t)=0$ for all $t \ge 0$. (Note that it is straightforward to check that if $\F_q \subset L^1(\rho_0)$ then $\F_q \subset L^1(\rho_t)$ for all $t>0$.) This shows that the EVI-flow leaves $\mathcal{F}_q^\perp$ invariant, and we can apply Theorem \ref{th:main}.

It remains to show that $\F_q^\perp$ is a geodesic space.
We first apply Proposition \ref{pr:Ylengthspace} to deduce that curves of finite length exist between every two points of $Y=\F_q^\perp$. We then apply Theorem \ref{th:Fgeodesic} with $\F=\F_q$, $\phi(x)=|x|^q$, and $c=\int_{\R^d} |y|^q\,N_{u,\Sigma} (\d y)$ to deduce both that $\D_Y=\W_{\F_q}$ and that $\F_q^\perp$ is a geodesic space.  
\end{proof}

\subsection{Logarithmic energy}

In this section we study the one-dimensional logarithmic energy $\e_{\mathrm{log}} : \ps_2(\R) \to (-\infty,\infty]$,
\begin{equation}
\e_{\mathrm{log}}(\rho) := -\int_\R\int_\R\log|x-y|\,\rho(\d x) \rho(\d y). \label{def:logenergy}
\end{equation}
By convention, $\log 0 := -\infty$. It was shown by Carrillo-Ferreira-Precioso \cite[Proposition 2.7]{carrillo2012mass} that $\e_{\mathrm{log}}$ is geodesically convex in $\ps_2(\R)$, but not strictly, because (like entropy) is it translation-invariant. This convexity is not at first obvious, because $x \mapsto -\log|x|$ is not globally convex in $\R$, but the key idea is that the convexity for $x>0$ suffices thanks to the monotonicity of optimal transport maps in one dimension. Here we show that, like entropy, this convexity becomes strong on Carlen-Gangbo spheres, and in Section \ref{se:BianeVoiculescu} we discuss the resulting strengthening of a transport inequality due to Biane-Voiculescu.

\begin{theorem} \label{th:logenergy}
For any $u \in \R$ and $\theta > 0$, $\e_{\mathrm{log}}$ is geodesically $(1/\theta)$-convex in $\mathcal{S}_{u,\theta}$.
\end{theorem}
\begin{proof}
We take $u=0$ for simplicity.
Let $\lambda=1/2\theta$, and consider the functional
\[
\e(\rho) := \frac12\e_{\mathrm{log}}(\rho) + \frac{\lambda}{2}\int_\R |x|^2\,\rho(\d x).
\]
We first argue that the $\mathrm{EVI}_\lambda$-flow for $\e$ exists in $\ps_2(\R)$ and identify it.
To achieve this, we use several results from \cite{carrillo2012mass}. First, Lemma 2.1 shows that $\e$ is lower semicontinuous in $\ps_2(\R)$, and Proposition 2.7 shows that $\e$ is $\lambda$-convex along generalized geodesics. The domain $D(\e)$ is dense in $\ps_2(\R)$, as it contains all bounded densities. As also $\e$ is clearly proper and coercive, we may apply Theorem 11.2.1 of \cite{ambrosio2008gradient}, which ensures the existence of the $\mathrm{EVI}_\lambda$-flow for $\e$ in $\ps_2(\R)$ from every starting point. Moreover, the $\mathrm{EVI}_\lambda$-flow coincides with the limit of the usual discretization (JKO) scheme. The convergence of the discretization was shown in Theorem 3.2 of \cite{carrillo2012mass} but to a different notion of gradient flow (Definition 3.1 therein), which must therefore coincide with the $\mathrm{EVI}_\lambda$-flow.
Finally, from Theorem 3.8 of \cite{carrillo2012mass} (or rather its proof) we deduce  that this $\mathrm{EVI}_\lambda$-flow $(\rho_t)_{t \ge 0}$ for $\e$ is a distributional solution of a nonlocal transport equation: for each smooth function $\varphi$ of compact support,
\begin{equation}
\frac{\d}{\d t} \int_\R \varphi\,\d \rho_t = -\lambda \int_\R x\varphi'(x)\,\rho_t(\d x) + \frac12\int_\R\int_\R\frac{\varphi'(x)-\varphi'(y)}{x-y}\,\rho_t(\d x)\rho_t(\d y). \label{eq:HilbertPDE}
\end{equation}

With the $\mathrm{EVI}_\lambda$-flow identified, we will now argue that it leaves $\mathcal{S}_{0,\theta}$ invariant. Assume $\rho_0 \in \mathcal{S}_{0,\theta}$. Because the flow $(\rho_t)_{t \ge 0}$ is contained in $\ps_2(\R)$ and continuous therein, a straightforward approximation argument extends \eqref{eq:HilbertPDE} to any smooth function $\varphi$ such that $|\varphi|$ has at most quadratic growth and $|\varphi'|$ has at most linear growth. Then, using $\varphi(x)=x$ in \eqref{eq:HilbertPDE}, we find
\begin{equation}
\frac{\d}{\d t}\int_\R x\,\rho_t(\d x) = -\lambda \int_\R x \,\rho_t(\d x).
\end{equation}
Since $\int_\R x\,\rho_0(\d x)=0$, we deduce that $\int_\R x\,\rho_t(\d x)=0$ for all $t > 0$. Use $\varphi(x)=x^2$ in \eqref{eq:HilbertPDE} to find
\begin{equation}
\frac{\d}{\d t}\bigg(\int_\R x^2\,\rho_t(\d x) - \frac{1}{2\lambda}\bigg) = -2\lambda\bigg( \int_\R x^2 \,\rho_t(\d x) - \frac{1}{2\lambda}\bigg).
\end{equation}
Since $\int_\R x^2\,\rho_0(\d x)=\theta=1/2\lambda$, the same is true for $\rho_t$ at all $t > 0$. We deduce that $\rho_t \in \mathcal{S}_{0,\theta}$ for all $t > 0$. Theorem \ref{th:main} now implies that $\e$ is geodesically $\lambda$-convex in $\mathcal{S}_{0,\theta}$. On $\mathcal{S}_{0,\theta}$, the functionals $\e$ and $\tfrac12\e_{\mathrm{log}}$ differ by an additive constant. Hence, $\e_{\mathrm{log}}$ is $2\lambda$-convex in $\mathcal{S}_{0,\theta}$, and the proof is complete.
\end{proof}

\begin{rem}
The equation \eqref{eq:HilbertPDE} is the distributional form of the PDE $\partial_t\rho = \partial_x((\lambda x -\pi\mathcal{H}(\rho))\rho)$, where $\mathcal{H}(\rho)$ is the Hilbert transform of $\rho$. This equation was studied in \cite{carrillo2012mass} (and prior work described therein) as a simplified model of of fluid mechanics. The same equation appears as the mean field limit of the Dyson-Ornstein-Uhlenbeck dynamics, as studied in \cite{rogers1993interacting} where uniqueness of \eqref{eq:HilbertPDE} was also proven.
\end{rem}

\begin{corollary}
Let $\theta > 0$ and $q \ge 2$.
Let $\F_q$ denote the set of polynomials on $\R^d$ of degree at most $q$ which have mean zero under the semicircle law $\sigma_\theta(\d x) = \frac{1}{2\pi\theta}\sqrt{(4\theta-x^2)_+}\,\d x$. 
Accordingly, define
\[
\mathcal{F}_q^\perp:= \bigg\{\rho \in \ps_2(\R^d) : \mathcal{F}_q \subset L^1(\rho), \ \int f\,\d\rho = 0, \ \forall f \in \mathcal{F}_q\bigg\}.
\]
Then $\e_{\mathrm{log}}$ is geodesically $(1/\theta)$-convex in $\mathcal{F}_q^\perp$.
\end{corollary}

\begin{proof}
This is a straightforward adaptation of the proofs of Theorems \ref{th:CGpolynomials} and \ref{th:logenergy}, once one observes that $\sigma_\theta$ is stationary for the PDE \eqref{eq:HilbertPDE}; see, e.g., Proposition 2.3 and Remark 2.4 of \cite{carrillo2012mass}.
\end{proof}

\subsection{Eigenspaces of diffusion operators}

The following setting generalizes the arguments behind Theorem \ref{th:CGentropy} and Theorem \ref{th:CGpolynomials}. The formal idea is quite simple: Consider a diffusion operator $L =  \Delta - \nabla V \cdot\nabla$ on a Riemannian manifold $M$, for some nice function $V$.
Let $\mathcal{F}$ be a set of eigenfunctions of $L$. That is, for each $f \in \mathcal{F}$ there exists $c_f \in \R$  such that $Lf=c_ff$. The EVI flow $(\rho_t)_{t \ge 0}$ for $\e(\rho)=H(\rho)+\int V\,\d \rho$ solves  the Fokker-Planck equation $\partial_t\rho = L^*\rho$. In particular, for all $f \in \F$, $\partial_t\int f \d\rho_t = \int L f \d\rho_t = c_f\int f \d\rho_t$. If $\rho_0$ satisfies $\int f\,\d\rho_0=0$, we deduce that $\int f \d\rho_t=0$ for all $t$. That is, a submanifold of $\ps_2(M)$ defined by constraints of this form, $\int f\d\rho=0$ for all $f \in \F$, is left invariant by the EVI flow for $\e$. This lets us apply the EVI invariance principle.
In Theorem \ref{th:CGpolynomials}, the relevant eigenfunctions were the Hermite polynomials up to degree $q$, or rather transformations thereof if $(u,\Sigma)\neq(0,I)$.

The only subtle point here is integrability. It is not a priori clear that $f \in L^1(\rho_0)$ implies $f \in L^1(\rho_t)$ for $t>0$. We check this in a few cases in the following theorem.

\begin{theorem}
Assume $(M,g,\pi)$ is a weighted $d$-dimensional Riemannian manifold satisfying the assumptions given in the paragraph preceding Proposition \ref{pr:Ylengthspace2}, including the  CD($\lambda$,$n$) condition \eqref{eq:CDCondition} for some $\lambda \in \R$ and $n \in (d,\infty]$.
Let $\F \subset L^2(\pi)$ be a set of non-constant eigenfunctions of $L = \Delta - \nabla V \cdot \nabla$, where $V: M \to \R$ is smooth.
Define $\e(\rho) = H(\rho) + \int V\,\d \rho$ and
\[
\F^\perp := \bigg\{ \rho \in \ps_2(M) : \mathcal{F} \subset L^2(\rho), \ \int f\,\d\rho =0, \ \forall f \in \mathcal{F}\bigg\}.
\]
\begin{enumerate}[(1)]
\item If $M$ is compact, then $\e$ is geodesically $\lambda$-convex on $\F^\perp$.
\item If $\lambda \ge 0$ and $n < \infty$, then $\e$ is geodesically $\lambda$-convex on $\F^\perp$.
\item If $\lambda > 0$ and $n < \infty$, and $M$ is compact, then $\F^\perp$ is a complete geodesic space.
\end{enumerate}
\end{theorem}
\begin{proof}
The eigenvalues of $L$ are functions in $H^1 (\pi)$ such that $L f = c_f f$.
Thus $L f$ is in $H^1_{\mathrm{loc}}$, and the ellipticity of $L$ classically implies that $f$ belongs to $H^{1+2}_{\mathrm{loc}}$, see e.g. \cite[Theorem 6.9]{grigoryan2009heat}. 
Iterating this argument, we deduce that $f$ is smooth.
From \cite[Corollary 1.5]{sturm2005convex}, the CD($\lambda$,$\infty$)  condition implies that $\e$ is geodesically $\lambda$-convex in $\ps_2 (M)$.
The EVI$_\lambda$ flow $(\rho_t)_{t \ge 0}$ for $\e(\rho)=H(\rho)+\int V\,\d \rho$ solves  the Fokker-Planck equation $\partial_t\rho = L^*\rho$. This is shown in \cite[Corollary 23.23]{villani2008optimal} in the compact case, and   \cite[Theorem 23.19]{villani2008optimal} in the non-compact case subject to an assumption of finite Fisher information which is checked as in the proof of Proposition \ref{pr:RiemannCD}.
As explained in the paragraph preceding the theorem statement, the Fokker-Planck equation leaves $\F^\perp$ invariant, as long as we can argue that $\F \subset L^1(\rho_t)$ for all $t > 0$ when $\rho_0 \in \F^\perp$. Once this is argued, Theorem \ref{th:main} applies to yield the geodesic $\lambda$-convexity of $\e$ on $\F^\perp$.

To check this integrability point, we treat the compact and non-compact cases separately. We have already noted that each $f \in \F$ must be smooth. Hence, if $M$ is compact, then each $f \in \F$ is bounded, and thus (1) is proven. In case (3), Proposition \ref{pr:RiemannCD} applies to show that any two points in $\F^\perp$ can be connected by a curve of finite length, and the boundedness of the eigenfunctions lets us apply Proposition \ref{pro:geodesicY} (with $\sigma$ being the weak topology) to deduce that $\F^\perp$ is a complete geodesic space. Hence, (3) is proven.

To treat the non-compact setting (2), we exploit properties of the heat kernel. As was explained in the proof of Proposition \ref{pr:RiemannCD}, there exists a smooth symmetric function $p_t (x,y)$ such that $\rho_t \ll \pi$ and $\d\rho_t/\d\pi(x)=\int_M p_t(x,y)\rho_0(\d y)$ for each $t > 0$. The assumptions $\lambda \ge 0$ and $n < \infty$ imply that $p_t(x,y)$ is bounded uniformly in $(x,y)$ for each fixed $t > 0$. Hence, the density $\d\rho_t/d\pi$ is bounded, and since $\F \subset L^1(\pi)$ we deduce that $\F \subset L^1(\rho_t)$ as well.
\end{proof}

\subsection{Couplings} \label{se:Couplings}

In this section we consider spaces of couplings. Let $\lambda \in \R$. For each $i=1,\ldots,n$, let $d_i \in \N$, and let $\mu_i \in \ps_2(\R^{d_i})$ have a strictly positive probability density such that $-\log\mu_i$ is $\lambda$-convex. 
Let $d=d_1+\cdots+d_n$ and write $x=(x_1,\ldots,x_n)$ for a generic element of $\R^d \cong \R^{d_1} \times \cdots \times \R^{d_n}$.
Let $\Pi\subset \ps_2(\R^d)$ denote the set of couplings. That is, a probability measure $\rho$ on $\R^d$ belongs to $\Pi$ if
\[
\int_{\R^d} f(x_i)\,\rho(\d x) = \int_{\R^{d_i}} f(x_i)\,\mu_i(\d x_i),
\]
for each $i$ and each bounded continuous function $f$ on $\R^d$.
The induced Wasserstein geometry of $\Pi$ was discussed recently in \cite{conforti2023projected}, in the context of constrained gradient flows associated with entropic optimal transport problems. Let us note that $\Pi$ is not a geodesically convex subset of $\ps_2(\R^d)$ unless it is a singleton \cite[Proposition 2.2]{conforti2023projected}.

\begin{theorem} \label{thm:CouplingHConvex}
In the setting above, $H$ is geodesically $\lambda$-convex in $\Pi$. Moreover, if $\lambda > 0$ and each $V_i$ has bounded second derivatives, then $\Pi$ is a complete geodesic space.
\end{theorem}
\begin{proof}
Let $\mu = \mu_1\otimes \cdots \otimes \mu_n$ denote the product measure.
Note for $\rho \in \Pi$ that
\[
H(\rho\,|\,\mu) = H(\rho) - \sum_{i=1}^nH(\mu_i),
\]
so that the $\lambda$-convexity of $H$ on $\Pi$ is equivalent to that  of $H(\cdot\,|\,\mu)$. 
Let $V_i=-\log\mu_i$ and note that $V(x)=V_1(x_1)+\cdots+V_n(x_n)=-\log\mu(x)$ is $\lambda$-convex.
The $\mathrm{EVI}_\lambda$-flow of $H(\cdot\,|\,\mu)$ on $\ps_2(\R^d)$ is governed by the Fokker-Planck equation \cite[Theorem 11.2.8]{ambrosio2008gradient},
\[
\partial_t \rho = \mathrm{div}(\rho\nabla V) + \Delta\rho.
\]
We next show that this evolution leaves $\Pi$ invariant, with an argument similar to  \cite[Lemma 3.2]{conforti2023projected}. Fix $i$ and a smooth test function $f$ on $\R^{d_i}$. Then
\begin{align}
\frac{\d}{\d t}\int_{\R^d} f(x_i)\,\rho_t(\d x) &= \int_{\R^d} \big(\Delta f(x_i) -\nabla f(x_i) \cdot \nabla V_i(x_i)\big) \,\rho_t(\d x).
\end{align}
Letting $\rho^i_t$ denote the $i$th marginal of $\rho_t$, this rewrites as
\begin{align}
\frac{\d}{\d t}\int_{\R^{d_i}} f \,\d\rho^i_t  &= \int_{\R^{d_i}} \big(\Delta f  -\nabla f  \cdot \nabla V_i \big) \,\d\rho^i_t .
\end{align}
This is the weak form of the Fokker-Planck equation $\partial_t\rho^i = \mathrm{div}(\rho^i\nabla V_i) + \Delta\rho^i$, for which $\mu^i$ is a stationary solution. If $\rho_0 \in \Pi$, then $\rho^i_0=\mu_i$, and the unique solution (e.g., by \cite[Theorem 9.6.3, Example 9.6.4]{bogachev2022fokker}) of this Fokker-Planck equation is constant in time, $\rho^i_t=\mu_i$ for all $t > 0$. Hence, $\rho_t \in \Pi$ for all $t$. Apply Theorem \ref{th:main} to deduce that $H(\cdot\,|\,\mu)$ and thus also $H$ are geodesically $\lambda$-convex in $\Pi$.

It remains to show that $\Pi$ is a geodesic space, in the case where $\lambda > 0$ and each $V_i$ has bounded second derivatives. We first note that the topological assumptions of Proposition \ref{pro:geodesicY} hold with $\sigma$ being the weak topology, as $\W_2$ is weakly lower semicontinuous and $\Pi$ is weakly compact. Moreover, we may apply Proposition \ref{pr:Ylengthspace2} to deduce that any pair of points in $\Pi$ may be connected by a finite-length curve contained in $\Pi$. Hence, Proposition \ref{pro:geodesicY} applies, and $\Pi$ is a geodesic space.
\end{proof}

\subsection{Incompressible optimal transport} \label{se:IncompressibleOT}

In this section, we apply our results to the setting of Brenier's incompressible optimal transport.
We first recall some of the history of this problem.

The Euler equations for incompressible fluids in the $d$-dimensional torus $\mathbb{T}^d$ read
\begin{equation} \label{eq:Euler}
\begin{cases}
\partial_t v + ( v \cdot \nabla ) v = - \nabla p, \\
\nabla \cdot v = 0,
\end{cases}
\end{equation}
the unknown being the velocity-pressure fields $(v,p)$ with $v : [0,1] \times \mathbb{T}^d \rightarrow \R^d$ and $p : [0,1] \times \mathbb{T}^d \rightarrow \R$.
If $v$ were smooth, it would produce a flow of diffeomorphisms $(X_t)_{t \geq 0}$ satisfying $X_0 = \mathrm{Id}$,
\[ \forall (t,a) \in [0,1] \times \mathbb{T}^d, \quad \dot{X}_t ( a ) = v_{t}(t,X_t(a)), \]
and preserving the Lebesgue measure thanks to the incompressibility condition $\nabla \cdot v = 0$.
Rewriting\footnote{We emphasize that this writing actually corresponds to the historical derivation of \eqref{eq:Euler} by Euler in his fundamental work \cite{EulerEuler}. See \cite[Chapter 2]{brenier2020examples} for more precisions.} \eqref{eq:Euler} in terms of $X_t$ yields
\begin{equation} \label{eq:EulerODE}
\begin{cases}
\ddot{X}_t (a) = - \nabla p (t, X_t (a)), \\
X_t \in \mathrm{SDiff},
\end{cases}    
\end{equation}
where $\mathrm{SDiff}$ is the space of measure-preserving diffeomorphisms on $\mathbb{T}^d$.
In the seminal work \cite{arnold1966geometrie}, Arnold interpreted \eqref{eq:EulerODE} as a geodesic equation on $\mathrm{SDiff}$, viewed as an infinite-dimensional Riemannian manifold with the metric inherited from the embedding into the flat space of $L^2$ vector fields.
This corresponds to the variational problem
\begin{equation} \label{eq:Arnold}
\inf_{(X_t)_{0 \leq t \leq 1} \subset \mathrm{SDiff} } \int_0^1 \int_{\mathbb{T}^d} \vert \dot{X}_t (a) \vert^2 \,\d a \d t, 
\end{equation}
with endpoints $X_0, X_1$ being imposed.
The pressure field $p$ can then be interpreted as a Lagrange multiplier related to the incompressibility constraint $X_t \# \mathrm{Leb} = \mathrm{Leb}$.

This constrained problem is delicate due to the lack of closedness and convexity for $\mathrm{SDiff}$. 
The pioneering works of Shnirelman \cite{shnirel1987geometry,shnirelman1994generalized} showed that the infimum in \eqref{eq:Arnold} is not attained in general for $d \geq 3$ and may be infinite for $d = 2$.

In \cite{brenier1989least}, Brenier introduced a convex relaxation of \eqref{eq:Arnold} for which existence of minimizers holds, by replacing the flow $( X_t )_{0 \leq t \leq 1}$ by a generalized flow in $\ps ( C( [0,1]; \mathbb{T}^d ))$.
In \cite{brenier1999minimal}, he introduced a further Eulerian-Lagrangian relaxation, by now known as the \emph{multiphase formulation}, which was convenient to prove existence, uniqueness and partial regularity for the pressure field.
We refer to \cite{ambrosio2010lecture,daneri2012variational} for lecture notes on these variational models as well as the textbook \cite[Chapters 2-4]{brenier2020examples}.

Given probability measures $\eta, \gamma \in \Pi ( \mathrm{Leb}, \mathrm{Leb} )$, the multiphase formulation
considers
\begin{equation} \label{eq:MultiPhase} \overline{\delta}(\eta,\gamma) := \inf \int_{\mathbb{T}^d} \int_0^1 \int_{\mathbb{T}^d} \vert v^a_t (x) \vert^2 \, c^a_t ( \d x) \d t \d a,  \end{equation}
where we minimize over continuous curves $( c_t )_{0 \leq t \leq 1}$ in $\Pi ( \mathrm{Leb}, \mathrm{Leb} )$ with $(c_0,c_1) = ( \eta, \gamma )$ with a disintegration $c_t ( \d x, \d a ) = c^a_t ( \d x ) \d a$ satisfying for a.e. $a \in \mathbb{T}^d$,
\begin{equation} \label{eq:PhaseTransp}
\partial_t c^a_t + \nabla_x \cdot ( c^a_t v^a_t ) = 0, 
\end{equation} 
in the sense of distributions, for $a$-dependent velocity vector fields $( v^a_t )_{0 \leq t \leq 1}$.
This formula amounts to superposing the Benamou-Brenier formula \eqref{intro:BenamouBrenier} for each $a$ under the global constraint $\int c^a_t \d a = \mathrm{Leb}$, justifying the terminology \emph{incompressible optimal transport}.
We refer to \cite[Section 4]{ambrosio2009geodesics} for the equivalence of the different relaxations of \eqref{eq:Arnold}.
Formal considerations \cite[Section 4]{brenier2003extended} suggest that any minimizer $(c,v)$ verifies $v^a_t = \nabla_x \varphi^a_t$ for a scalar potential $\varphi^a$ satisfying the Hamilton-Jacobi equation
\[ \partial_t \varphi^a_t (x) + \frac{1}{2} \vert \nabla_x \varphi^a_t (x) \vert^2 = - p_t (x), \]
the pressure field $( p_t )_{0 \leq t \leq 1}$ being a Lagrange multiplier for the constraint on the $x$-marginal of $c_t$ in the minimization \eqref{eq:MultiPhase}; the $a$-marginal does not matter since it is automatically preserved by \eqref{eq:PhaseTransp}.
A formal computation then yields
\[ \frac{\d^2}{\d t^2} \int_{\mathbb{T}^d} c^a_t (x) \log c^a_t (x) \d x = \int_{\mathbb{T}^d} \big(\Delta_x p_t (x) + \vert \nabla_x^2 \varphi^a_t (x) \vert^2 \big) \d c^a_t (x). \]
Defining the averaged entropy
\[ \H(c_t) := \int_{\mathbb{T}^d} H(c^a_t)\, \d a, \]
the fact that $\int c^a_t \d a = \mathrm{Leb}$
then suggests that $\frac{\d^2}{\d t^2} \H ( c_t ) \geq 0$.
Thus, \cite[Section 4]{brenier2003extended} conjectured that $\H$ is convex along the minimizers of \eqref{eq:MultiPhase}.
This conjecture was proved by Lavenant in \cite{lavenant2017time}, under some restriction later removed by Baradat and Monsaingeon \cite{baradat2020small}. 

In \cite{brenier1999minimal}, Brenier proved existence of minimizers for \eqref{eq:MultiPhase}, which are non unique in general.
In spite of this non-uniqueness, Brenier proved existence and uniqueness for the pressure field as a distribution.
More precisely, $\nabla p$ was shown to be a measure in $\M ( (0,1) \times \mathbb{T}^d )$, which only depends on the endpoints $(\eta,\gamma)$.
This result was strengthened in \cite{ambrosio2008regularity} where Ambrosio and Figalli established that $\nabla p \in L^2((0,1); \M (\mathbb{T}^d ))$, and then $p \in L^2 ( (0,1),L^{d/(d-1)} ( \mathbb{T}^d))$. 
However, the optimal regularity of the pressure remains an open question, despite a conjecture by Brenier \cite{brenier2013remarks} that $p$ is semiconcave.

Thus, the available regularity on $p$ is far from sufficient for making rigorous the above computation of $\tfrac{\d^2}{\d t^2} \H ( c_t )$.
The proofs in \cite{lavenant2017time,baradat2020small} circumvent this issue by using variational tools similar to the flow interchange method developed by Matthes, McCann and Savaré \cite{matthes2009family} to propagate bounds on geodesically convex functions along iterates of the JKO scheme.
More precisely, \cite{lavenant2017time} approximated \eqref{eq:MultiPhase} by a time-discrete problem that is suitable to apply the flow interchange method, and the result followed from the continuous limit.
And \cite{baradat2020small} directly worked at the continuous level, by regularizing geodesics via heat flow. 
This method was further developed in \cite{monsaingeon2023dynamical} to study abstract Schrödinger problems in metric spaces.

Let us now detail how Theorem \ref{th:main} allows for a simple proof of the convexity of $\H$ along the minimizers of \eqref{eq:MultiPhase}.
We consider the space $X = \Pi ( \cdot, \mathrm{Leb} )$ of probability measures on $\mathbb{T}^d \times \mathbb{T}^d$ whose second marginal is Lebesgue measure, endowed with the modified Wasserstein distance $\D_X$ given by 
\[ \D_X^2(\eta,\gamma) = \int_0^1 \W_2^2 ( \eta^a, \gamma^a ) \,\d a. \]
This distance appeared first in \cite[Proposition 12.4.6]{ambrosio2008gradient} to describe geometric tangent cones in the Wasserstein space, and its various properties (topology, completeness, separability, geodesics, gradient flows...) were systematically studied in \cite{bashiri2020gradient} and \cite{peszek2023heterogeneous}.\footnote{Strictly speaking, \cite{bashiri2020gradient} considers a slightly different setting with $\mathbb{T} \times \R$ in place of $\mathbb{T}^d \times \mathbb{T}^d$. 
However, the adaptation to our present compact setting raises no challenge  \cite[Remark 2.4]{bashiri2020gradient}.}
The metric subspace that we consider is $Y = \Pi ( \mathrm{Leb}, \mathrm{Leb})$.
From Theorem \ref{th:Fgeodesic}, we deduce that the induced length distance on $\Pi ( \mathrm{Leb}, \mathrm{Leb})$ corresponds to $\overline{\delta}$ given by \eqref{eq:MultiPhase}.
Finiteness of the distance and existence of geodesics is given by \cite{brenier1999minimal}.

\begin{theorem} \label{th:ConvexOIT}
The functional $\H$ is geodesically convex in $(\Pi(\mathrm{Leb},\mathrm{Leb}),\overline{\delta})$. 
\end{theorem}
\begin{proof}
The fact that $\H$ is geodesically $0$-convex in $(X, \D_X)$ and generates a $\mathrm{EVI}_0$ flow is provided by \cite[Lemma 3.28 and Theorem 3.27]{bashiri2020gradient},
and the idea behind this is as follows: Starting from any $c_0 = c^a_0 ( \d x ) \d a$ in $\Pi ( \cdot, \mathrm{Leb} )$, the usual heat flow provides the $\mathrm{EVI}_0$ flow $( c^a_t )_{t \geq 0}$ for entropy on $(\ps_2 ( \mathbb{T}^d ),\W_2)$, for each $a$, satisfying
\[ \frac{1}{2} \frac{\d^+}{\d t} \W^2_2 (c^a_t,\mu) \leq H ( \mu ) - H ( c^a_t ), \quad \forall \mu \in \ps ( \mathbb{T}^d), \ H(\mu) < \infty, \  t >0. \]
Up to measurability issues justified in \cite{bashiri2020gradient}, integrating over $a$ then yields the desired $\mathrm{EVI}_0$ flow $(c_t)_{t \ge 0}$ in $(X, \D_X)$. 
From the characterization \cite[Theorem 3.41]{bashiri2020gradient}, such a gradient flow $(c_t)_{t \geq 0}$ is a solution of
\[ \partial_t c_t = \Delta_{x} c_t,   \]
in the sense of distributions, which is again formally clear as a consequence of the heat flow $\partial_t c^a_t = \Delta_{x} c_t^a$ for each $a$.
If $c_0$ belongs to $Y=\Pi ( \mathrm{Leb}, \mathrm{Leb})$, then $c_t$ remains in $Y$ at every time $t > 0$ because the Lebesgue measure is invariant under heat flow. 
The geodesic convexity of $\H$ in $(Y,\D_Y)=(\Pi(\mathrm{Leb},\mathrm{Leb}),\overline{\delta})$ now follows from  Theorem \ref{th:main}.
\end{proof}

\section{Strengthened transport inequalities and large deviations} \label{se:transport-LD}

The Talagrand inequality and other \emph{transport inequalities} are intimately connected to concentration of measure and large deviations. 
In this section we begin an exploration of similar connections with the inequalities arising from Proposition \ref{pr:MetricTalagrand}, which might best be called \emph{constrained transport inequalities}.
We refer to \cite{gozlan2010transport} for a thorough survey of classical transport inequalities, and we discuss here only one particular connection with large deviations for which we have a clear picture. An intriguing direction for future research would be to develop a fuller picture of the implications or equivalences of constrained transport inequalities with concentration of measure.

We work in this section on the Carlen-Gangbo sphere $\mathcal{S}:=\mathcal{S}_{0,1} \subset \ps_2(\R^d)$ defined in Section \ref{se:CarlenGangbo}. 
We discuss in Section \ref{se:TalagrandLD} a strengthening of Talagrand's inequality for the Gaussian measure and in Section \ref{se:BianeVoiculescu} a strengthening of the Biane-Voiculescu inequality for the semicircle law. 

\subsection{Talagrand's inequality} \label{se:TalagrandLD}

Carlen and Gangbo already observed (in the paragraph after \cite[Theorem 3.6]{carlen2003constrained}) how Theorem \ref{th:CGentropy} leads to the inequality
\begin{equation}
\D^2_{\mathcal{S}}(\gamma,\mu) \le 2 H(\mu\,|\,\gamma), \quad \forall \mu \in \mathcal{S}, \label{CG-Talagrand}
\end{equation}
where $\gamma=N_{0,I}$ is the standard Gaussian measure on $\R^d$. This also follows from our Proposition \ref{pr:MetricTalagrand}. Recall that $\D_{\mathcal{S}} \ge \W_2$ on $\mathcal{S}$, so that this inequality is stronger than the restriction of Talagrand's inequality to $\mathcal{S}$.
Moreover, by Proposition \ref{pr:HWI}, we have the following new variant of the HWI inequality:
\begin{equation}
H( \mu\,|\,\gamma ) - H(\nu\,|\,\gamma) \leq \D_{\mathcal{S}}( \mu,\nu) I(\mu\,|\,\gamma) - \frac{1}{2} \D_{\mathcal{S}}^2 (\mu,\nu), \quad \forall \mu,\nu \in \mathcal{S}, \ H(\nu\,|\,\gamma) < \infty.
\end{equation}
It was shown in \cite[Equation (3.33)]{carlen2003constrained} that the intrinsic distance on $\mathcal{S}$ is given explicitly by
\begin{equation}
\D_{\mathcal{S}}(\mu,\nu) = 2\sqrt{d}\arcsin\bigg(\frac{\W_2(\mu,\nu)}{2\sqrt{d}}\bigg).
\end{equation}
(Note that in \cite{carlen2003constrained} the definition of $\W_2^2$ is half of the usual definition, and we adopt the latter; hence our definition of $\D_{\mathcal{S}}^2$ is twice the definition of \cite{carlen2003constrained}.)
Using this formula, we may rewrite the inequality \eqref{CG-Talagrand} as
\begin{equation}
\alpha(\W_2^2(\gamma,\mu)) \le H(\mu\,|\,\gamma), \quad  \forall \mu \in \mathcal{S}, \label{CG-Talagrand-alpha}
\end{equation}
where we define a convex function $\alpha : \R \to [0,\infty]$ by
\[
\alpha(x) := \begin{cases}
2d\arcsin^2(\sqrt{x}/2\sqrt{d}) &\text{if } 0 \le x \le 2\sqrt{d} \\ \infty &\text{otherwise}.
\end{cases}
\]
In fact, in the next proposition, we show that the inequality \eqref{CG-Talagrand-alpha} extends in a non-trivial way from $\mathcal{S}$ to its weak closure, by taking lower semicontinuous envelopes. Let us abbreviate $m_1(\mu)=\int x\,\mu(\d x)$ and  $m_2(\mu)=\int|x|^2\,\mu(\d x)$, when these quantities are well defined.

\begin{proposition} \label{pr:lscenvelope}
For every $\mu \in \ps(\R^d)$ we have
\begin{equation*}
\alpha(\W_2^2(\gamma,\mu)) \le  I(\mu) := \begin{cases}
H(\mu\,|\,\gamma) + \frac{1}{2}(d-m_2(\mu)) &\text{if } \mu \in \ps_2(\R^n), \ m_2(\mu) \le d, \ m_1(\mu)=0 \\ \infty &\text{otherwise}. \end{cases}
\end{equation*}
Moreover, $I$ is the lower semi-continuous envelope (with respect to weak convergence) of the function  $J : \ps(\R^d) \to [0,\infty]$ defined by
\[
J(\mu) := \begin{cases} H(\mu\,|\,\gamma) &\text{if } \mu \in \mathcal{S} \\ \infty &\text{otherwise}. \end{cases}
\]
\end{proposition}

The proof is postponed to the end of the section.
Remarkably, the same functional $I$ appears as the rate function governing the large deviations of the empirical measure of $d$-coordinate blocks of a high-dimensional random vector drawn uniformly from a sphere. To state this precisely, we introduced some notation. For each $n \in \N$, let $U^n$ be uniformly distributed on the $dn$-dimensional sphere $\sqrt{dn}\,\mathbb{S}^{dn-1}$ of radius $\sqrt{dn}$. Write $U^n=(U^n_1,\ldots,U^n_n)$ where $U^n_i=(U^n_{i,j})_{j=1,\ldots,d}$. That is, the vector $U^n_i$ denotes the $i$th block of $d$ coordinates. Consider the empirical measure 
\[
L_n(U^n) = \frac{1}{n}\sum_{i=1}^n\delta_{U^n_i},
\]
viewed as a random element $\ps(\R^d)$.  Let us also define the (continuous) \emph{centering map} $\mathsf{C} : \ps_1(\R^d) \to \ps_1(\R^d)$ which translates a measure to have mean zero; i.e., $\mathsf{C} \mu := (\mathsf{Id}-m_1(\mu))\#\mu$. 
It is well known that $L_n(U^n)$ converges weakly in probability to $\gamma$, which is equivalent to the better-known statement that $(U^n_1,\ldots,U^n_k)$ converges in law as $n\to\infty$ to $\gamma^{\otimes k}$, for each fixed $k \in \N$. Since $\E|U^n_i|^2=d$, this convergence upgrades from the weak topology to $\ps_p(\R^d)$, for any $p \in [1,2)$. Hence, $\mathsf{C}L_n(U^n)$ converges in probability in $\ps_p(\R^d)$ to $\gamma$.
The following is the associated large deviations principle.

\begin{theorem} \label{th:LDP}
For any $p \in [1,2)$, the sequence $(\mathsf{C} L_n(U^n))$ satisfies a large deviation principle in $\ps_p(\R^d)$ with good rate function $I$. That is, for any closed set $F \subset \ps_p(\R^d)$ and open set $G \subset \ps_p(\R^d)$,
\begin{align}
\limsup_{n\to\infty}\frac{1}{n}\log\P\big( \mathsf{C}L_n(U^n) \in F\big) &\le -\inf_{\mu \in F}I(\mu), \\
\liminf_{n\to\infty}\frac{1}{n}\log\P\big( \mathsf{C}L_n(U^n) \in G\big) &\ge -\inf_{\mu \in G}I(\mu).
\end{align}
\end{theorem}

In the case $d=1$, and without the centering operation, this was shown in \cite[Theorem 6.6]{arous2001aging}, and we refer also to \cite[Proposition 3]{kim2018conditional} for a more self-contained proof where they also treated uniform measures on $\ell^p$ balls.
The idea of the proof is as follows: Let $X_i$ denote iid standard Gaussians in dimension $n$, and let $X^m=(X_1,\ldots,X_m)$. Construct the uniform on the sphere via $U^n_i=\sqrt{dn}X^n_i/|X^n|$. The pair $(L_n(X^n),\frac{1}{dn}\sum_{i=1}^n|X_i|^2)$ obeys a large deviation principle in $\ps(\R^d) \times \R$, according Cram\'er's theorem in topological vector spaces, and we apply the contraction principle to obtain the result for $\mathsf{C} L_n(U^m)$.

The appearance of $I$ as a rate function lets us characterize the transport inequality \eqref{CG-Talagrand-Jbar} in terms of large deviation probabilities. Indeed, we follow a recipe developed by Gozlan and L\'eonard \cite{gozlan2007large} to deduce the following:

\begin{proposition} \label{pr:LDP-transport}
Let $r > 0$.
For any bounded continuous functions $f$ and $g$ on $\R^d$ satisfying $f(x) - g(y) \le |x-y|^2$, we have
\begin{equation}
\limsup_{n\to\infty}\frac{1}{n}\log\P\bigg(\frac{1}{n}\sum_{i=1}^nf\Big(U^n_i - \frac{1}{n}\sum_{j=1}^nU^n_j\Big) \ge \langle\gamma,g\rangle + r\bigg) \le - \alpha(r). \label{pr:LDbound1}
\end{equation}
Moreover,  
\begin{equation}
\limsup_{n\to\infty}\frac{1}{n}\log\P\big(\W_1( \mathsf{C} L_n(U^n),\gamma) \ge r\big) \le - \alpha(r^2). \label{pr:LDbound2}
\end{equation}
\end{proposition}
\begin{proof}
Using the large deviations upper bound followed by Kantorovich duality, the left-hand side of \eqref{pr:LDbound1} is bounded from above by
\begin{align*}
-\inf\Big\{I(\mu) : \mu \in \ps(\R^d), \, \langle \mu,f\rangle-\langle\gamma,g\rangle \ge r\Big\} \le -\inf\Big\{I(\mu) : \mu \in \ps(\R^d), \, \W_2^2(\mu,\gamma) \ge r\Big\}.
\end{align*}
Using Proposition \ref{pr:lscenvelope}, this is further bounded by $-\alpha(r)$ because $\alpha(\cdot)$ is increasing. Similarly, the left-hand side of \eqref{pr:LDbound2} is bounded from above by
\begin{align*}
-\inf\Big\{I(\mu) : \mu \in \ps(\R^d), \, \W_1(\mu,\gamma) \ge r\Big\},
\end{align*}
which by Proposition \ref{pr:lscenvelope} and Jensen is bounded by $-\alpha(r^2)$.
\end{proof}

Notice that \eqref{pr:LDbound2} implies that for any $1$-Lipschitz function $f$ on $\R^d$,
\begin{equation}
\limsup_{m\to\infty}\frac{1}{n}\log\P\bigg(\frac{1}{n}\sum_{i=1}^n f\Big(U^n_i - \frac{1}{n}\sum_{j=1}^nU^n_j\Big)   \ge \langle\gamma,f\rangle + r\bigg) \le - \alpha(r^2). \label{pr:LDbound3}
\end{equation}
Noting that $\arcsin(x)\ge x$ for $0 \le x \le 1$, we have $\alpha(r^2) \ge r^2/2$ for $r >0$. Proposition \ref{pr:LDP-transport} thus shows that the large deviations of $\mathsf{C}L_n(U^n)$ are more rare than the analogous ones for the empirical measure of iid Gaussians.
As in \cite[Theorem 2]{gozlan2007large}, the arguments of Proposition \ref{pr:LDP-transport} can be reversed, by using the \emph{lower bound} of the large deviation principle. The statement that \eqref{pr:LDbound1} holds for all such $f$ and $g$ and $r$ is \emph{equivalent} to the inequality \eqref{CG-Talagrand-Jbar}. Similarly, the weaker inequality $\alpha(\W_1^2(\gamma,\cdot)) \le I$ is equivalent to the statement that \eqref{pr:LDbound3} holds for all $r$ and all $1$-Lipschitz  $f$.
Analogously, the classical Talagrand inequality $\W_2^2(\mu,\gamma) \le 2H(\mu\,|\,\gamma)$ for all $\mu \in \ps(\R^d)$ is equivalent to the statement that 
\begin{equation}
\limsup_{n\to\infty}\frac{1}{n}\log\P\bigg(\frac{1}{n}\sum_{i=1}^nf(X_i) \ge \langle\gamma,g\rangle + r\bigg) \le - r/2,
\end{equation}
for all $r>0$ and bounded continuous functions $f$ and $g$ on $\R^d$ satisfying $f(x) - g(y) \le |x-y|^2$.
And the weaker \emph{T1 inequality} $\W_1^2(\mu,\gamma) \le 2H(\mu\,|\,\gamma)$ for all $\mu \in \ps(\R^d)$ is equivalent to the statement that 
\begin{equation}
\limsup_{n\to\infty}\frac{1}{n}\log\P\bigg(\frac{1}{n}\sum_{i=1}^nf(X_i) \ge \langle\gamma,f\rangle + r\bigg) \le - r^2/2,
\end{equation}
for all $r>0$ and $1$-Lipschitz functions $f$ on $\R^d$.

It is worth comparing Proposition \ref{pr:LDP-transport} to known subgaussian estimates for $U^n$.
For example, for $d=1$ and the identity function $f(x)=x$, there is a well-known concentration inequality \cite[Proof of Theorem 3.4.5]{vershynin2025high}
\[
\P\bigg(\frac{1}{n}\sum_{i=1}^n U^n_i \ge r\bigg) = \P(U^n_1/\sqrt{n} \ge r ) \le (1- r^2)^{(n-1)/2}, \quad 0 \le r \le 1.
\]
Since $-(1/2)\log(1-r^2) \ge \alpha(r^2)$ for $d=1$, this is a sharper tail bound than \eqref{pr:LDbound3} (although they are not directly comparable because of the centering). However, this is concentration estimate just for the sample average, whereas Proposition \ref{pr:LDP-transport} is at a functional level.

It is natural to wonder if the developments of this section can be generalized beyond the example of the Carlen-Gangbo sphere. Let $\mathcal{M} \subset \ps_2(\R^d)$ and $c > 0$, and suppose $\pi \in \ps_2(\R^d)$ has been shown to satisfy $\D_{\mathcal{M}}^2(\mu,\pi) \le 2cH(\mu\,|\,\pi)$ for all $\mu \in \mathcal{M}$. In light of the Gibbs conditioning principle, we speculate that this inequality is related to the large deviation behavior of the empirical measure of iid samples of $\pi$ \emph{conditioned to belong to $\mathcal{M}$}. To avoid conditioning on a measure-zero set, one should more rigorously perform a set enlargement, as is standard in Gibbs conditioning arguments. The connection with the previous example is that the uniform measure on a sphere provides a version of the conditional law of the Gaussian given its norm. It appears to be challenging to formulate a rigorous picture in the general case, so we leave this for future work.

\begin{proof}[Proof of Proposition \ref{pr:lscenvelope}]
We first note that the inequality \eqref{CG-Talagrand-alpha} trivially implies
\begin{equation}
\alpha(\W_2^2(\gamma,\mu)) \le J(\mu), \quad  \forall \mu \in \ps(\R^d). \label{CG-Talagrand-J}
\end{equation}
Next, let us define the lower semicontinuous envelope of $J$,
\[
\overline{J}(\mu) := \liminf_{\mu' \to \mu}J(\mu').
\]
Equivalently, this is the pointwise supremum of all weakly lower semicontinuous functions dominated  by $J$. 
Because $\W_2(\gamma,\cdot)$ is lower semicontinuous and $\alpha$ is increasing on $[0,\infty)$, the inequality \eqref{CG-Talagrand-J} is equivalent to
\begin{equation}
\alpha(\W_2^2(\gamma,\mu)) \le \overline{J}(\mu), \quad  \forall \mu \in \ps_2(\R^d). \label{CG-Talagrand-Jbar}
\end{equation}
Hence, the proof will be complete once we show that $I \equiv \overline{J}$.

We first show that $I$ is lower semicontinuous. To see this, note first that a simple calculation shows 
\[
I(\mu) = \begin{cases}
H(\mu) + \frac{d}{2} + \frac{d}{2}\log(2\pi)  &\text{if } \mu \in \ps_2(\R^d), \ m_2(\mu) \le d, \ m_1(\mu)=0 \\ \infty &\text{otherwise}. \end{cases}
\]
It is well known that $H(\cdot)$ is lower semicontinuous on $\ps_1(\R^d)$. The topologies of $\ps_1(\R^d)$ and $\ps(\R^d)$ agree on the set $\{\mu \in \ps_2(\R^d): m_2(\mu) \le 1, \, m_1(\mu)=0\}$, which is weakly compact. (In fact, it is the closure of $\mathcal{S}$ in $\ps_1(\R^d)$.) It follows that $I$ is lower semicontinuous. Since $I \le J$ pointwise, it follows that $I \le \overline{J}$ as well, by definition of $\overline{J}$.

To show that $I \ge \overline{J}$, we will fix $\mu \in \ps(\R^d)$ and show that there exists a sequence $(\mu_k)$ converging weakly to $\mu$ and satisfying $\limsup_kJ(\mu_k) \le I(\mu)$.
We may assume that $I(\mu) < \infty$, so that $m_2(\mu) \le d$ and $m_1(\mu)=0$.
For $k \in \N$, let $\eta_k\in\ps_2(\R^d)$ denote the law of  the random vector $(\epsilon(Z_i+\sqrt{k-1}))_{i=1,\ldots,d}$, where $Z_i$ are iid standard normals (in dimension 1) and $\epsilon$ is uniform on $\{-1,+1\}$. Note that $m_1(\eta_k)=0$ and $m_2(\eta_k)=dk$.
Define $\mu_k = c_k\mu + (1-c_k)\eta_k$, where we choose $c_k = (dk-d)/(dk-m_2(\mu))$ so that $m_2(\mu_k)=d$. Of course $m_1(\mu_k)=0$ as well, so $\mu_k \in \mathcal{S}$.
Noting that $c_k \in [0,1]$, the convexity of entropy yields
\begin{align*}
J(\mu_k) &= H(\mu_k) + \frac{d}{2} + \frac{d}{2}\log(2\pi) \le c_kH(\mu) + (1-c_k)H(\eta_k) + \frac{d}{2} + \frac{d}{2}\log(2\pi).
\end{align*}
By convexity and translation-invariance of entropy, we have $H(\eta_k) \le H(\gamma) < \infty$ for each $k$.
Since $c_k\to 1$ as $k\to\infty$, it follows that
\begin{align*}
\limsup_{k\to\infty} J(\mu_k) &\le H(\mu) + \frac{d}{2} + \frac{d}{2}\log(2\pi) \\
    &= H(\mu) + \frac12\int|x|^2\,\mu(\d x) + \frac{d}{2}\log(2\pi) + \frac{1}{2}(d-m_2(\mu)) \\
    &= H(\mu\,|\,\gamma) + \frac{1}{2}(d-m_2(\mu)).
\end{align*}
The right-hand side is $I(\mu)$, and the proof is complete.
\end{proof}

\subsection{Biane-Voiculescu inequality} \label{se:BianeVoiculescu}

The logarithmic energy $\e_{\mathrm{log}}$ defined in \eqref{def:logenergy} is ubiquitous in random matrix theory. It was shown in \cite{arous1997large} that 
\[
I_{\mathrm{log}}(\mu) :=\e_{\mathrm{log}}(\rho) + \frac12\int_\R x^2\,\rho(\d x) - \frac34
\]
is the rate function governing the large deviations of the empirical eigenvalue distribution for Wigner matrices.
The unique minimizer of $I_{\mathrm{log}}$ is given by the semicircle law,
\[
\sigma(\d x) = \frac{1}{2\pi}\sqrt{(4-x^2)_+}\,\d x.
\]
The minimal value is $I_{\mathrm{log}}(\sigma)=0$, because $\int x^2 \,\sigma(\d x) = 1$ and $\e_{\mathrm{log}}(\sigma)=1/4$.

Biane and Voiculescu showed in \cite[Theorem 2.8, Remark 2.9]{biane2001free} that
\begin{equation}
\W_2^2(\sigma,\mu) \le 2I_{\mathrm{log}}(\mu), \quad \forall \mu \in \ps_2(\R).
\end{equation}
This can be viewed as the analogue of Talagrand's inequality in the theory of \emph{free probability}, and we refer to \cite[Section 12]{gozlan2010transport} for further discussion and references.
Combining Theorem \ref{th:logenergy} with Proposition \ref{pr:MetricTalagrand}, we obtain the following reinforcement of this inequality:

\begin{proposition} \label{pr:BianeVoiculescu}
For all $\mu \in \mathcal{S}_{0,1}$ we have
\begin{equation}
4\arcsin^2\Big(\frac12 \W_2(\mu,\sigma)\Big) \le 2I_{\mathrm{log}}(\mu).
\end{equation}
\end{proposition}
\begin{proof}
Theorem \ref{th:logenergy} (with $\theta=1$) shows that the functional $I_{\mathrm{log}}(\cdot)$ is $1$-convex in $\mathcal{S}_{0,1}$. Because $\sigma$ is the unique minimizer of $I_{\mathrm{log}}$ on $\ps_2(\R)$, it is also the unique minimizer on $\mathcal{S}_{0,1}$. Using Proposition \ref{pr:MetricTalagrand} along with $I_{\mathrm{log}}(\sigma)=0$ proves the claim.
\end{proof}

In the spirit of Section \ref{se:TalagrandLD}, it is natural to expect the inequality \eqref{pr:BianeVoiculescu} to relate to the aforementioned large deviations principle for Wigner matrices. The Gaussian Unitary Ensemble (GUE) is the probability measure on the space $\mathcal{H}_n$ of $n \times n$ Hermitian matrices whose density with respect to Lebesgue measure is proportional to $H \mapsto \exp(- \mathrm{tr}(H^2)/2)$. The associated \emph{fixed trace ensemble}, formally obtained by conditioning on  $\mathrm{tr}(H^2)=n^2$, has the same expression for its density relative to the surface measure on $\{H \in \mathcal{H}_n : \mathrm{tr}(H^2)=n^2\}$. Let $M_n$ be a random matrix drawn from this fixed trace ensemble, and let $\lambda^n_1 < \cdots < \lambda^n_n$ denote the eigenvalues of $M_n/\sqrt{n}$. Consider the centered empirical measure
\[
\mathsf{C}L_n = \frac{1}{n}\sum_{i=1}^n\delta_{\lambda^n_i - \frac{1}{n}\sum_{j=1}^n\lambda^n_j}.
\]
We conjecture that $\mathsf{C}L_n$ satisfies a large deviation principle in $(\ps_1(\R),\W_1)$ with good rate function
\[
\mu \mapsto \begin{cases}
    I_{\mathrm{log}}(\mu) + \frac12\big(1 - m_2(\mu)\big) &\text{if } m_1(\mu)=0, \ m_2(\mu)\le 1, \\ \infty &\text{otherwise},
\end{cases}
\]
where again $m_p(\mu)=\int_\R x^p\,\mu(\d x)$ for $p \ge 1$. It is not difficult to show (similarly to Proposition \ref{pr:lscenvelope}) that this conjectural rate function is the weakly lower-semicontinuous envelope of 
\[
\mu \mapsto \begin{cases}
    I_{\mathrm{log}}(\mu) &\text{if } \mu \in \mathcal{S}_{0,1} \\ \infty &\text{otherwise}.
\end{cases}
\]
With this large deviation principle in hand, it would follow as in Proposition \ref{pr:LDP-transport} that
\[
\limsup_{m\to\infty}\frac{1}{n}\log\P\big(\W_1( \mathsf{C} L_n,\sigma) \ge r\big) \le - 2\arcsin^2(r/2),
\]
for $0 < r \le 2$, with the right-hand side replaced by $-\infty$ for $r > 2$.
A simpler consequence in terms of Lipschitz test functions may be stated analogously to \eqref{pr:LDbound3}.
This would mean that, similarly to the setting of Section \ref{se:TalagrandLD}, conditioning the trace of the GUE improves the concentration profile from $r^2/2$ to $2\arcsin^2(r/2)$.

\appendix

\section{An alternate proof of the EVI invariance principle} \label{se:alternateproof}

Here we sketch an alternative proof of the geodesic convexity in $Y$ from Theorem \ref{th:main}, based on the methods of \cite{baradat2020small} and especially \cite{monsaingeon2023dynamical}. We mostly just repeat the argument of \cite[Theorem 4.8]{monsaingeon2023dynamical},
with one point of departure. Let $(\gamma_t)_{t \in [0,1]}$ be a $Y$-geodesic with endpoints assumed to be in $D(\e)$ without loss of generality. We will use semigroup notation to denote EVI flows: For each $y\in Y \cap D(E)$, let $(\mathsf{S}_ty)_{t > 0}$ denote the EVI$_\lambda$-flow of $\e$ in $(X,\D_X)$ started from $y$. Fix $\theta \in (0,1)$ and define
\[
h_\epsilon(t) := \begin{cases}
\epsilon t/\theta &\text{if } t \in [0,\theta] \\ \epsilon (1-t)/(1-\theta) &\text{if } t \in [\theta,1]. \end{cases}
\]
Let $\gamma^\epsilon_t := \mathsf{S}_{h_\epsilon(t)}\gamma_t$. This curve is a perturbation of the geodesic $\gamma$ \emph{in the direction of the gradient flow}, and it featured prominently in the works \cite{baradat2020small,monsaingeon2023dynamical}.  
We will take as an input a key calculation of \cite[Proposition 3.11]{monsaingeon2023dynamical}, after discarding the $|\partial \e|$ term therein:
\begin{equation}
\frac12|\dot\gamma^\epsilon|_t^2 + h'_\epsilon(t)\frac{\d}{\d t}\e(\gamma^\epsilon_t) \le \frac12 e^{-2\lambda h_\epsilon(t)}|\dot\gamma|_t^2, \quad a.e. \ t \in (0,1). \label{ineq:MTV1}
\end{equation}
The derivation of this inequality comes from an application of the definition of EVI$_\lambda$ flows with carefully chosen reference points \cite[Lemma 3.6]{monsaingeon2023dynamical}.
Now, since $\gamma$ is a $Y$-geodesic, we have $|\dot\gamma|_t=\D_Y(\gamma_0,\gamma_1)=:D$ for a.e.\ $t \in [0,1]$. The main assumption that $\mathsf{S}_t$ leaves $Y$ invariant now comes into play: It implies that the curve $(\gamma^\epsilon_t)_{t \in [0,1]}$ is contained in $Y$, and so the definition of $\D_Y$ yields
\[
D^2 \le \Big(\int_0^1|\dot\gamma^\epsilon|_t\,\d t\Big)^2 \le  \int_0^1|\dot\gamma^\epsilon|_t^2\,\d t.
\]
Integrating \eqref{ineq:MTV1}, we find
\begin{align}
0 &\le \frac12 \int_0^1|\dot\gamma^\epsilon|_t^2\,\d t - \frac{D^2}{2} \le -\int_0^1 h'_\epsilon(t)\frac{\d}{\d t}\e(\gamma^\epsilon_t)\,\d t  + \frac{D^2}{2}I_\epsilon,
\end{align}
where we set $I_\epsilon := \int_0^1\big(e^{-2\lambda h_\epsilon(t)}-1\big)\,\d t$. Splitting the first integral into $[0,\theta]$ and $[\theta,1]$,
\begin{align}
0 &\le -\frac{\epsilon}{\theta}\int_0^\theta \frac{\d}{\d t}\e(\gamma^\epsilon_t)\,\d t + \frac{\epsilon}{1-\theta}\int_\theta^1 \frac{\d}{\d t}\e(\gamma^\epsilon_t)\,\d t + \frac{D^2}{2}I_\epsilon \\
    &= \frac{\epsilon}{\theta}\big(\e(\gamma_0)-\e(\gamma^\epsilon_\theta)\big) + \frac{\epsilon}{1-\theta}\big(\e(\gamma_1)-\e(\gamma^\epsilon_\theta)\big) + \frac{D^2}{2}I_\epsilon.
\end{align}
Multiply by $\theta(1-\theta)$ and compute $\lim_{\epsilon \downarrow 0}I_\epsilon/\epsilon = -\lambda$ to get
\begin{align}
\liminf_{\epsilon \downarrow 0}\e(\gamma^\epsilon_\theta) \le (1-\theta)\e(\gamma_0) + \theta \e(\gamma_1) - \frac{\lambda D^2}{2}\theta(1-\theta).
\end{align}
Finally, since $h_\epsilon(t)\to 0$ pointwise, we have $\gamma^\epsilon_\theta = \mathsf{S}_{h_\epsilon(\theta)}\gamma_\theta \to \gamma_\theta$ as $\epsilon \downarrow 0$. The very definition that an EVI flow $\mathsf{S}_ty$ \emph{starts at the point $y$} also required that $\e(y) \le \liminf_{t\downarrow 0}\e(\mathsf{S}_ty)$. Hence, $\e(\gamma_\theta) \le \liminf_{t\downarrow 0}\e(\gamma^\epsilon_\theta)$, and so $\e$ is geodesically $\lambda$-convex in $Y$. \hfill\qedsymbol

\subsection*{Acknowledgements}
The authors would like to thank Jose Carrillo for sharing and discussing the paper \cite{carrillo2012mass}.

\printbibliography

\noindent\textsc{Louis-Pierre Chaintron, DMA, École normale supérieure, Université PSL, CNRS, 75005 Paris, France}\\
\emph{Email address:} \texttt{louis-pierre.chaintron@ens.fr}

\vspace{1em}

\noindent\textsc{Daniel Lacker, Department of Industrial Engineering and Operations Research, Columbia University}\\
\emph{Email address:} \texttt{daniel.lacker@columbia.edu}

\end{document}